\documentclass{amsart}

\usepackage[utf8]{inputenc}
\usepackage{amsmath}
\usepackage{amsthm}
\usepackage{amsfonts}
\usepackage{amssymb}

\usepackage{fullpage}
\usepackage{mathtools}
\usepackage{stmaryrd}
\usepackage{mathrsfs}

\usepackage{listings}
\usepackage{bbm}
\usepackage{accents}
\usepackage{tikz}
\usepackage{tikz-cd}
\usepackage{quiver}
\usepackage{caption}

\usepackage{multirow}

\usepackage{adjustbox}
\usepackage{censor}

\usepackage{enumitem}

\usepackage{epigraph}

\usepackage[bookmarks,hidelinks,pdfencoding=unicode]{hyperref}
\usepackage[capitalize]{cleveref}
\usepackage{animate}

\newtheorem{thm}{Theorem}[section]

\newtheorem{prop}[thm]{Proposition}
\newtheorem{lem}[thm]{Lemma}

\newtheorem{fact}[thm]{Fact}

\newtheorem{quest}[thm]{Question}

\theoremstyle{definition}
\newtheorem{defn}[thm]{Definition}

\theoremstyle{remark}

\newcommand{\Rb}{\mathbb{R}}

\newcommand{\Zb}{\mathbb{Z}}
\newcommand{\Qb}{\mathbb{Q}}

\newcommand{\Fc}{\mathcal{F}}

\newcommand{\Rc}{\mathcal{R}}

\newcommand{\Lc}{\mathcal{L}}

\newcommand{\Kc}{\mathcal{K}}

\newcommand{\Cb}{\mathbb{C}}
\newcommand{\Fb}{\mathbb{F}}

\newcommand{\mfr}{\mathfrak{m}}

\newcommand{\onebf}{\mathbf{1}}

\makeatletter \DeclareRobustCommand{\cset}{\@ifstar\star@cset\normal@cset}
\newcommand{\star@cset}[1]{{\left\llbracket#1\right\rrbracket}}
\newcommand{\normal@cset}[2][]{{\mathopen{#1\llbracket}#2\mathclose{#1\rrbracket}}}
\makeatother

\newcommand{\DCF}{\ifmmode\mathsf{DCF}_0\else$\mathsf{DCF}_0$\fi}
\newcommand{\CODF}{\ifmmode\mathsf{CODF}\else$\mathsf{CODF}$\fi}
\newcommand{\ODF}{\ifmmode\mathsf{ODF}\else$\mathsf{ODF}$\fi}
\newcommand{\ACF}{\ifmmode\mathsf{ACF}\else$\mathsf{ACF}$\fi}

\DeclareMathOperator{\acl}{acl}

\DeclareMathOperator{\dom}{dom}

\newcommand{\IZF}{\ifmmode\mathsf{IZF}\else$\mathsf{IZF}$\fi}
\newcommand{\CZF}{\ifmmode\mathsf{CZF}\else$\mathsf{CZF}$\fi}
\newcommand{\ZF}{\ifmmode\mathsf{ZF}\else$\mathsf{ZF}$\fi}
\newcommand{\ZFC}{\ifmmode\mathsf{ZFC}\else$\mathsf{ZFC}$\fi}
\newcommand{\AC}{\ifmmode\mathsf{AC}\else$\mathsf{AC}$\fi}
\newcommand{\AD}{\ifmmode\mathsf{AC}\else$\mathsf{AD}$\fi}
\newcommand{\BZ}{\ifmmode\mathsf{BZ}\else$\mathsf{BZ}$\fi}
\newcommand{\GCH}{\ifmmode\mathsf{GCH}\else$\mathsf{GCH}$\fi}
\newcommand{\PA}{\ifmmode\mathsf{PA}\else$\mathsf{PA}$\fi}
\newcommand{\DC}{\ifmmode\mathsf{DC}\else$\mathsf{DC}$\fi}
\newcommand{\MP}{\ifmmode\mathsf{MP}\else$\mathsf{MP}$\fi}
\newcommand{\CT}{\ifmmode\mathsf{CT}\else$\mathsf{CT}$\fi}
\newcommand{\PAx}{\ifmmode\mathsf{PAx}\else$\mathsf{PAx}$\fi}
\newcommand{\VL}{\ifmmode{\mathsf{V}=\mathsf{L}}\else$\mathsf{V}=\mathsf{L}$\fi}

\DeclareMathOperator{\Sh}{Sh}

\newlength{\savedparindent}
\AtBeginDocument{\setlength{\savedparindent}{\parindent}}

\def\Ind{\setbox0=\hbox{$x$}\kern\wd0\hbox to 0pt{\hss$\mid$\hss}
\lower.9\ht0\hbox to 0pt{\hss$\smile$\hss}\kern\wd0}

\def\Notind{\setbox0=\hbox{$x$}\kern\wd0\hbox to 0pt{\mathchardef
\nn=12854\hss$\nn$\kern1.4\wd0\hss}\hbox to
0pt{\hss$\mid$\hss}\lower.9\ht0 \hbox to 0pt{\hss$\smile$\hss}\kern\wd0}

\DeclareMathOperator{\cl}{cl}
\DeclareMathOperator{\tint}{int}

\DeclareMathOperator{\RO}{RO}

\newcommand{\dense}{\textnormal{dense}}
\newcommand{\Kden}{\Kc_{\dense}}
\newcommand{\Rden}{\Rc_{\dense}}

\begin{document}
\title{A natural haystack of differentially closed fields}

\address{Department of Mathematics \\
  Iowa State University \\
  396 Carver Hall \\
  411 Morrill Road \\
  Ames, IA 50011, USA}
\author{James E. Hanson}
\email{jameseh@iastate.edu}
\date{\today}


\begin{abstract}
  In this partially expository paper, we present a novel construction of differentially closed fields of characteristic $0$: Let $\Kden$ be the differential ring of all meromorphic functions whose domain is a (not necessarily connected) dense open subset of $\Cb$ modulo agreement on dense open sets (i.e., $f$ and $g$ are considered equal if there is a dense open $U \subseteq \Cb$ such that $f|_U = g|_U$). We show that every ring ideal of $\Kden$ is a differential ideal and that for every maximal ideal $\mfr$, the quotient $\Kden/\mfr$ is a differentially closed field. We also show that $\Kden/\mfr$ is saturated and has cardinality of the continuum, implying that any two such quotients are isomorphic as differential fields.

  We then discuss how to motivate this construction in terms of set-theoretic forcing, Boolean-valued models, and $\neg\neg$-sheaves on $\Cb$, taking the opportunity to present an impressionistic expository account of these ideas.

  Finally, we discuss some immediate generalizations of this construction involving the real and $p$-adic numbers and ask some questions about them.

\end{abstract}

\maketitle

\begin{figure}[h]
  \centering
  \includegraphics[width=\textwidth]{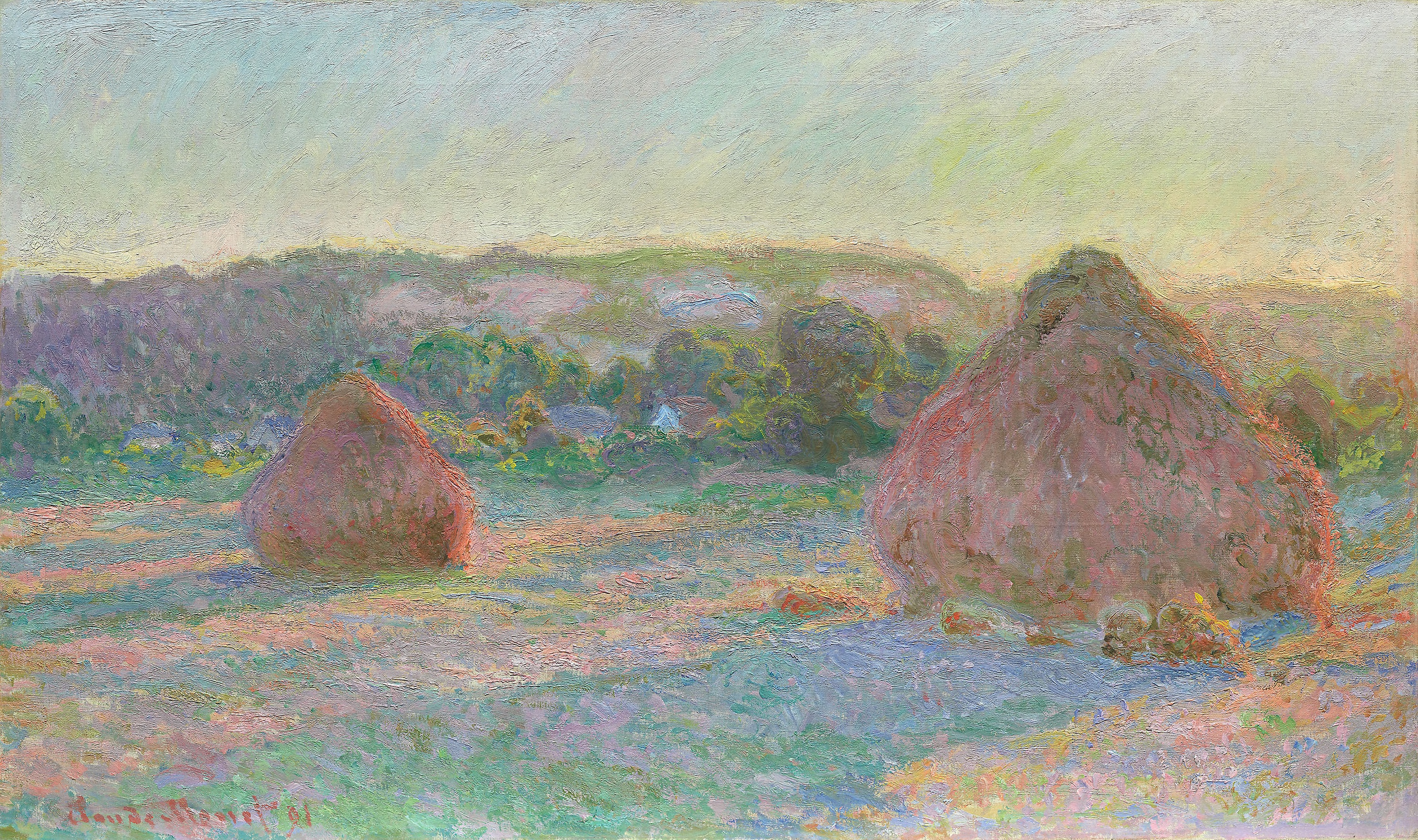}
  \caption{Monet, Claude. \emph{Stacks of Wheat (End of Summer)}. 1890-91. Oil on canvas.}
  \label{fig:Monet}
\end{figure}

\section*{Introduction}
\label{sec:intro}

Differentially closed fields have historically been hard to pin down. As Spodzieja writes in \cite{Spodzieja2020}, ``Despite a fairly long period of study of differential algebras, it is difficult to indicate papers where natural examples of differentially closed fields are given.'' As such, the conventional wisdom is often that there simply are no `natural' examples. Like algebraically closed fields of positive characteristic, the only obvious way to exhibit such a thing is building the field by hand inductively as the union of a tower of abstract extensions. Spodzieja's somewhat technical construction in \cite{Spodzieja2020} seems to be the only published method of building differentially closed fields using ideas that are not entirely algebraic or model-theoretic in nature.

As opposed to positive characteristic fields, which are seemingly purely algebraic objects, there is a large family of naturally occurring characteristic $0$ differential fields. For any connected open set $U \subseteq \Cb$ (or more generally for any Riemann surface), the ring of meromorphic functions on $U$, $\Kc(U)$, is a differential field. And similarly, for any $a \in \Cb$, the stalk $\Kc_a$ of germs of meromorphic functions at $a$ is a differential field. In \cite{Seidenberg1969}, Seidenberg showed that every countable characteristic $0$ differential field embeds into $\Kc_a$ (for any $a$), establishing a connection between abstract differential fields and complex analysis. He did this using the following lemma, which is conceptually central to this paper.

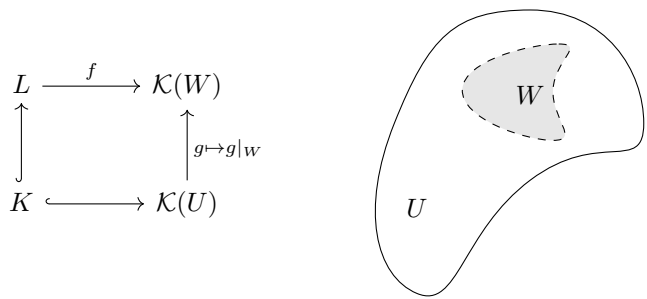
\begin{figure}[h]
  \centering
  \begin{tikzpicture}

    \draw (-2.75,-0.25) node {\begin{tikzcd}[row sep=large,column sep=large]
        L & {\Kc(W)} \\
        K & {\Kc(U)}
        \arrow["f", from=1-1, to=1-2]
        \arrow[hook, from=2-1, to=1-1]
        \arrow[hook, from=2-1, to=2-2]
        \arrow["{g\mapsto g|_W }"'{pos=0.45}, from=2-2, to=1-2]
      \end{tikzcd}};

    \begin{scope}[shift={(1,0)},scale=1.5]
      \draw plot [smooth cycle,tension=1.1] coordinates {(-0.2,0) (1,1) (2,0) (1,-0.5) (0,-1.5)};
      \draw (0,-0.75) node {$U$};
      \draw [dashed,fill=black,fill opacity=0.1] plot [smooth cycle,tension=1.1] coordinates {(0.4,0.3) (1.25,0.7) (1.2,0.3) (1.2,-0.15)};
      \draw (1,0.25) node {$W$};
    \end{scope}
  \end{tikzpicture}
  \caption{Seidenberg's embedding lemma (\cref{lem:Seidenberg-lemma}).}
  \label{fig:Seidenberg-1}
\end{figure}

\begin{lem}[Seidenberg {\cite{Seidenberg1969}}]\label[lem]{lem:Seidenberg-lemma}
  Let $U$ be a non-empty connected open subset of $\Cb$ and let $K \subseteq \Kc(U)$ be a finitely generated differential field. If $L \supseteq K$ is a finitely generated abstract differential field extension of $K$, then there exists a non-empty connected open set $W\subseteq U$ and a differential field embedding $f:L\to \Kc(W)$ such that the diagram in Figure~\ref{fig:Seidenberg-1} commutes.
\end{lem}

Compare this to Kolchin's notion of \emph{universality} for differential fields (a mild strengthening of differential closure\footnote{Model-theoretically, universal differential fields are precisely $\aleph_0$-saturated differentially closed fields.}).

\begin{defn}[Kolchin {\cite[Ch.~III.7]{kolchin1973differential}}]\label[defn]{defn:universal-field}
  A characteristic $0$ differential field $F$ is \emph{universal} if whenever $K \subseteq F$ is a finitely generated differential subfield and $L \supseteq K$ is a finitely generated differential field extension of $K$, there exists a differential field embedding $f : L \to F$ such that the following diagram commutes:
  \[\begin{tikzcd}
      L & \\
      K & F
      \arrow["f", from=1-1, to=2-2]
      \arrow[hook, from=2-1, to=1-1]
      \arrow[hook, from=2-1, to=2-2]
    \end{tikzcd}\]
\end{defn}

What \cref{lem:Seidenberg-lemma} tells us is that the sheaf $\Kc$ of meromorphic functions behaves like a universal characteristic $0$ differential field in some `dynamic' sense, a sense in which we're only concerned with finitely many meromorphic functions at a time and we are willing to pass to smaller open domains (without much control over which smaller open domains we pass to at any given time).

Despite this, no fields of the form $\Kc(U)$ or $\Kc_a$ are universal or even differentially closed. For any $a \in U$, there is no antiderivative of $\frac{1}{z-a}$ in $\Kc(U)$ or in $\Kc_a$ (since $\log(z-a)$ has a branch cut). This obstruction suggests a childish question:

\vspace{0.5em}
\begin{center}
  If $a$ was a complex number that magically wasn't the pole of any meromorphic function, would $\Kc_a$ be a differentially closed field?
\end{center}
\vspace{0.5em}
Now, while this idea is obviously vacuous nonsense, what we will show in this paper is that something very similar can be formalized. In a still fantastical but somewhat more accurate manner of speaking, our main result can be interpreted as saying that if $a$ is a complex number that is an element of every dense open subset of $\Cb$, then $\Kc_a$ is a differentially closed field.

The presentation of the main construction---the one described in the abstract---is given in a self-contained way in Sections~\ref{sec:main} and \ref{sec:models} and Appendix~\ref{ap:saturated}. This might not immediately resemble the idea described in the previous paragraph aside from the vague commonality of dense open sets, but we will back up and justify this perspective in \cref{sec:forcing}, the expository body of the paper. \cref{sec:generalizations-and-questions} contains a variety of generalizations and questions suggested by this construction.

\section{The ring $\Kden$}
\label{sec:main}

\begin{defn}\label[defn]{defn:Kden-ring}
  Given meromorphic functions $f : U \to \Cb$ and $g : V \to \Cb$, with $U,V\subseteq \Cb$ (not necessarily connected) dense open sets, $f$ and $g$ \emph{agree on a dense open set} if there is a dense open set $W \subseteq U \cap V$ such that $f|_W = g|_W$. (This is equivalent to saying that $f$ and $g$ agree on the intersection of their domains.) It is easy to see that this is an equivalence relation that respects differential ring operations.  Given $f : U \to \Cb$, we will write $[f]$ for equivalence class of $f$ under this equivalence relation.

  Let $\Kden$ be the set of equivalence classes of this equivalence relation endowed with the obvious differential ring operations:
  \begin{itemize}
  \item $0 = [0]$.
  \item $1 = [1]$.
  \item $[f]+[g] = [(f+g)|_{\dom(f)\cap\dom(g)}]$.
  \item $[f][g] = [f g |_{\dom(f)\cap\dom(g)}]$.
  \item $D[f]=[Df]$.
  \end{itemize}
\end{defn}

It is immediate that $\Kden$ is a (commutative) differential ring and that $\Kden$ has no non-trivial nilpotent elements. It is also easy to see that the analogous ring defined in terms of holomorphic functions is isomorphic to $\Kden$.

Recall that $R$ is a \emph{von Neumann regular ring} if for every $r \in R$, there is an $s \in R$ such that $r = rsr$. Moreover, such an $s$ is called a \emph{weak inverse} of $r$.
It is fairly easy to see that $\Kden$ is a von Neumann regular ring and that for any $[f] \in \Kden$,
\[
	[f]^\dagger = \left[
    \begin{cases}
      \frac{1}{f(z)} & f(z) \neq 0 \\
      0 & f~\textnormal{vanishes in some neighborhood of}~z
    \end{cases}
	\right]
\]
is a weak inverse of $[f]$.
It follows from the general theory of von Neumann regular rings (but is also obvious in this particular case) that every ideal of $\Kden$ is generated by some (possibly infinite) family of idempotent elements and every finitely generated ideal is generated by a single idempotent element. In particular, the principal ideal $\langle [f] \rangle$ is also generated by
\[
  [f][f]^\dagger = \left[ \begin{cases}
    1 & f(z) \neq 0 \\
    0 & f~\textnormal{vanishes in some neighborhood of}~z
  \end{cases} \right],
\]
which is idempotent. Moreover, it is clear that every idempotent element of $\Kden$ is of this form (i.e., is (the equivalence class of) a $\{0,1\}$-valued function). We can say a little bit more. Given any open set $U \subseteq \Cb$, we define
\[
  \onebf_U(z) = \begin{cases}
    1 & z \in U \\
    0 & z \in \tint(\Cb \setminus U)
  \end{cases}.
\]
Note that for any open $U,V \subseteq \Cb$, $[\onebf_U] = [\onebf_V]$ if and only if $\cl U=\cl V$. Every idempotent element of $\Kden$ is of the form $[\onebf_U]$ and these correspond precisely to regular open subsets of $\Cb$. Moreover, $1-[\onebf_U] = [\onebf_{\tint(\Cb\setminus U) }]$ for any $U$.

\begin{prop}\label[prop]{prop:diff-ideal}
  Every ring ideal of $\Kden$ is a differential ideal.
\end{prop}
\begin{proof}
  For any $[f] \in \Kden$, we have that $D[f] = [Df][f]^\dagger[f]$ (since $f'(z)$ vanishes on any open set on which $f(z)$ vanishes). It follows immediately that $D[f]$ is an element of any ideal containing $[f]$.
\end{proof}

\cref{prop:diff-ideal} implies that every quotient ring of $\Kden$ has a canonical differential ring structure, so in particular $\Kden/\mfr$ is a differential field for any maximal ideal $\mfr$.

\section{$\Kden/\mfr$ is differentially closed}
\label{sec:models}

For the remainder of this section and the next, fix a maximal ideal $\mfr$ in $\Kden$. Our goal is to show that $\Kden/\mfr$ is a differentially closed field. It is also clearly a $\Cb$-algebra and so has characteristic $0$. Therefore we just need to show that $\Kden/\mfr$ satisfies the differential closure condition: For any differential polynomials $p$ and $q$ where the order of $p$ is greater than that of $q$, there is an $x \in \Kden/\mfr$ such that $p(x) = 0$ and $q(x)\neq0$.

We will write $[[f]]$ for the element of $\Kden/\mfr$ corresponding to $[f] \in \Kden$.

\begin{lem}\label[lem]{lem:basic-stuff} $ $
  \begin{enumerate}
  \item\label{either-or} For any open sets $U,V \subseteq \Cb$, if $U\cap V = \varnothing$, then either $[\onebf_U] \in \mfr$ or $[\onebf_V] \in \mfr$.
  \item\label{zero-char} For any $[f] \in \Kden$, $[[f]] = 0$ if and only if there is an open set $U$ with $[\onebf_U] \in \mfr$ such that for all $z \in \dom(f)$, if $f(z) \neq 0$, then $z \in U$.
  \item\label{nonzero-char} For any $[f] \in \Kden$, $[[f]] \neq 0$ if and only if there is an open set $U$ with $[\onebf_U] \in \mfr$ such that for all $z \in \dom(f)$, if $f$ vanishes on some neighborhood of $z$, then $z \in U$.
  \end{enumerate}
\end{lem}
\begin{proof}
  (\ref{either-or}) Since $[\onebf_U][\onebf_V] = 0 \in \mfr$, the result follows from the fact that $\mfr$ is maximal and therefore prime.

  (\ref{zero-char}) If there is such a $U$, then clearly $[f]\in\mfr$ and thereby $[[f]] = 0$. On the other hand if $[[f]] = 0$, then $[f] \in \mfr$, so $\langle [f] \rangle \subseteq \mfr$. $\langle [f] \rangle$ is also generated by $[f][f]^\dagger$, which is equal to $[\onebf_{\{z : f(z) \neq 0\}}]$.

  (\ref{nonzero-char}) Let $V = \{z : f(z) \neq 0\}$ and $U = \{z : f~\text{vanishes on some neighborhood of}~z\}$. Clearly $V\cap U = \varnothing$. By (\ref{zero-char}), $[\onebf_V] \notin \mfr$, so by (\ref{either-or}), $[\onebf_U] \in \mfr$.
\end{proof}

The next fact is closely related to \cref{lem:Seidenberg-lemma} and follows from some basic results of complex analysis and ordinary differential equations. We will write $\Kc(U)$ for the field of meromorphic functions on $U\subseteq \Cb$.

Recall that the \emph{order} of a differential polynomial $p(x)$ is the highest order of a derivative of $x$ appearing in $p(x)$ (with the convention that the order of $p(x)$ is $-1$ if it is a constant polynomial).

\begin{fact}\label[fact]{fact:indep-solution}
  Let $U$ be a non-empty connected open subset of $\Cb$ and let $p(x)$ be a differential polynomial with coefficients in $\Kc(U)$. Let $K \subseteq \Kc(U)$ be a countable differential field containing the coefficients of $p(x)$.
  \begin{enumerate}
  \item\label{poly} If $p(x)$ has order $0$ (i.e., is a non-constant ordinary polynomial), then there is a dense open $V \subseteq U$ and a meromorphic function $f$ on $V$ such that $p(f) = 0$.
  \end{enumerate}
  Now assume that the order of $p(x)$ is $n \geq 1$.
  \begin{enumerate}
    \setcounter{enumi}{1}
  \item\label{generic-point} There is a point $z_0 \in U$ that is not a pole or zero of any non-zero element of $K$.
  \item\label{solution} For any such $z_0$, there is an open neighborhood $V \ni z_0$ with $V \subseteq U$ and a non-empty open set $W \subseteq \Cb^n$ such that for any $\vec{a} \in W$, there is an $f_{\vec{a}} \in \Kc(V)$ such that
    \begin{itemize}
    \item $f_{\vec{a}}^{(i)}(z_0) = a_i$ for each $i < n$ and
    \item $p(f_{\vec{a}}) = 0$ on $V$.
    \end{itemize}
  \item\label{alg-indep} For any such $f(z)$, if the set $\{f(z_0),f'(z_0),f^{(2)}(z_0),\dots,f^{(n-1)}(z_0)\}$ is algebraically independent over the field $\{g(z_0):g \in K\}$, then $g(f) \neq 0$ for every differential polynomial $g$ of order less than $n$ with coefficients from $K$.
  \end{enumerate}
\end{fact}

\begin{defn}
  Say that a \emph{generic differential polynomial}\footnote{This is mildly non-standard in that we're allowing zero monomials.} is a multivariable differential polynomial \[p(x;y_0,\dots,y_n)= \sum_{i \leq n}y_iq_i(x)\] such that
  \begin{itemize}
  \item each $q_i(x)$ is either a differential monomial with no coefficient or $0$ and
  \item no two non-zero $q_i$'s are equal.
  \end{itemize}
\end{defn}
\begin{sloppypar}
Note that for any differential polynomial $q(x) \in R\{x\}$, there is a generic differential polynomial $p(x;y_0,\dots,y_n)$ and $a_0,\dots,a_n \in R$ such that $q(x) = p(x;a_0,\dots,a_n)$. Moreover, for any two differential polynomials $q_0(x),q_1(x) \in R\{x\}$, there are generic differential polynomials $p_0(x;y_0,\dots,y_n)$ and $p_1(x;y_0,\dots,y_n)$ and $a_0,\dots,a_n \in R$ such that $q_i(x) = p_i(x;a_0,\dots,a_n)$ for both $i < 2$. (This is the reason why we allow the differential monomials to be $0$.) The following lemma is easy to verify.
\end{sloppypar}

\begin{lem}\label[lem]{lem:diff-poly-temp}
  For any generic differential polynomial $p(x;y_0,\dots,y_m)$ and meromorphic functions $f_0,\dots,f_m$ on $U$, if each $f_i$ is nowhere identically vanishing\footnote{We say that a function $f$ is \emph{nowhere identically vanishing} on $U$ if there does not exist a non-empty open $V \subseteq U$ on which $f$ vanishes identically.} on $U$, then for every non-empty $V \subseteq U$, the order of $p(x;f_0|_V,\dots,f_m|_V)$ is equal to the order of $p(x;f_0,\dots,f_m)$. \qed
\end{lem}

\begin{lem}\label[lem]{lem:Seidenberg-dense}
  Let $U$ be a non-empty open subset of $\Cb$, let $h_0,\dots,h_m$ be nowhere identically vanishing meromorphic functions on $U$, and let $p(x;y_0,\dots,y_m)$ and $q(x;y_0,\dots,y_m)$ be non-zero generic differential polynomials such that the order of $p(x;h_0,\dots,h_m)$ is greater than the order of $q(x;h_0,\dots,h_m)$. Then there exists an open $W \subseteq U$ dense in $U$ and a meromorphic function $f$ on $W$ such that
  \begin{itemize}
  \item $p(f;h_0|_W,\dots,h_m|_W) = 0$ and
  \item $q(f;h_0|_W,\dots,h_m|_W)$ is nowhere identically vanishing on $W$.
  \end{itemize}
\end{lem}

\begin{proof}
  First assume that the order of $p$ is $0$ (so in particular $p$ is an ordinary polynomial and $q$ is a non-zero constant). Let $(U_i)_{i< k}$ (with $k \leq \omega$) be an enumeration of the connected components of $U$. For each $i < k$, apply \cref{fact:indep-solution} part (\ref{poly}) to $p(x;h_0|_{U_i},\dots,h_m|_{U_i})$ to get an open set $W_i\subseteq U_i$ dense in $U_i$ and a meromorphic function $f_i$ on $W_i$ such that $p(f_i;h_0|_{U_i},\dots,h_m|_{U_i}) = 0$. Let $f$ be the unique meromorphic function on $W = \bigcup_{i < k} W_i$ satisfying that $f|_{W_i} = f_i$ for each $i < k$. Since $q$ has order $-1$, it must be a constant, so $f$ satisfies the required conditions by construction.

  Now assume that the order of $p$ is at least $1$. Let $(V_i)_{i < \omega}$ be an enumeration of a basis of open subsets of $U$ such that for every $n$, $\bigcup_{i > n} V_i$ is dense in $U$.
We will define the following data by induction: a sequence $(W_n)_{n < \omega}$ of pairwise disjoint connected non-empty open subsets of $U$ and  $(f_n)_{n < \omega}$ such that $f_n$ is meromorphic on $W_n$ and satisfies that $p(f_n;h_0|_{W_n},\dots,h_m|_{W_n}) = 0$ and $q(f_n;h_0|_{W_n},\dots,h_m|_{W_n}) \neq 0$.

  \begin{sloppypar}
  At stage $n$, if $V_n \setminus \bigcup_{j < n} W_j$ is empty, let $W_n$ be the empty set and $f_n$ be the empty function. Otherwise, let $O_n$ be a connected component of $V_n \setminus \cl \bigcup_{j < n} W_j$. Note that since $h_0,\dots,h_m$ are nowhere identically vanishing on $U$, we have by \cref{lem:diff-poly-temp} that the order of $p(x;h_0|_{O_n},\dots,h_m|_{O_n})$ is still greater than the order of $q(x;h_0|_{O_n},\dots,h_m|_{O_n})$. Therefore, by \cref{fact:indep-solution}, we can find a connected open subset $W_n \subseteq O_n$ and a meromorphic function $f_n$ on $W_n$ such that $p(f_n;h_0|_{W_n},\dots,h_m|_{W_n}) = 0$ and $q(f_n;h_0|_{W_n},\dots,h_m|_{W_n}) \neq 0$. Note that since $W_n$ is connected, it must be the case that $q(f_n;h_0|_{W_n},\dots,h_m |_{W_n})$ does not vanish identically on any non-empty open subset of $W_n$. (See Figure~\ref{fig:density}.)
  \end{sloppypar}

  Let $W = \bigcup_{n < \omega}W_n$. Let $f = \bigcup_{n < \omega}f_n$ be the unique meromorphic function on $W$ satisfying that $f |_{ W_n} = f_n$ for each $n < \omega$. $f$ now satisfies the required conditions by construction.
\end{proof}

\begin{figure}
  \centering
  \input{figures/animation.tex}
  \caption{The proof of \cref{lem:Seidenberg-dense}. At each stage, we look at the next element $V_n$ of our basis, apply \cref{fact:indep-solution} to get an open set $V_n$ and a local solution $f_n$ to the differential equation, and restrict to a maximal $W_n \subseteq V_n$ that does not intersect with previously chosen $W_m$'s.}
  \label{fig:density}
\end{figure}

\begin{thm}\label[thm]{thm:models-DCF}
  For any differential polynomials $p$ and $q$ with coefficients in $\Kden/\mfr$, if $p$ has higher order than $q$, then there is a $[[g]] \in \Kden/\mfr$ such that $p([[g]]) = 0$ and $q([[g]]) \neq 0$.

  In particular, $\Kden/\mfr$ is a differentially closed field.
\end{thm}
\begin{proof}
  Fix generic differential polynomials $p_0$ and $q_0$ such that $p(x) = p_0(x;[[h_0]],\dots,[[h_n]])$ and $q(x) = q_0(x;[[h_0]],\dots,[[h_n]])$. We may assume without loss of generality that $[[h_i]] \neq 0$ for all $i \leq n$.

  For each $i \leq n$, apply \cref{lem:basic-stuff} part (\ref{nonzero-char}) to $h_i$ to produce an open set $V_i$. Let $V = \bigcup_{i \leq n} V_i$. It is immediate that $[\onebf_V] \in \mfr$ and that for any $z$, if some $h_i$ vanishes identically in some neighborhood of $z$, then $z \in V$.

  Let $U = \bigcap_{i \leq n}\dom(h_i) \cap \tint(\Cb \setminus V) $, and note that $1-[\onebf_U] = [\onebf_V] \in \mfr$. Note also that by construction, each $h_i$ is nowhere identically vanishing on $U$, so we can apply \cref{lem:Seidenberg-dense} to obtain an open set $W \subseteq U$ dense in $U$ and a meromorphic function $f$ on $W$ such that $p_0(f;h_0|_W,\dots,h_n|_W) = 0$ and such that $q_0(f;h_0|_W,\dots,h_n |_W)$ is nowhere identically vanishing on $W$. Define
  \[
    g(z) = \begin{cases}
      f(z) & z \in \dom(f) \\
      0 & z \in \tint(\Cb \setminus \dom(f))
    \end{cases}.
  \]
  Note that, by construction, $p_0(g;h_0,\dots,h_n)$ vanishes on $W$ and $q_0(g;h_0,\dots,h_n)$ vanishes identically nowhere on $W$. Since $1-[\onebf_W] = 1-[\onebf_U] \in \mfr$, this implies that $p([[g]]) = p_0([[g]];[[h_0]],\dots,[[h_n]]) = 0$ and $q([[g]]) = q_0([[g]];[[h_0]],\dots,[[h_n]]) \neq 0$ by \cref{lem:basic-stuff} parts (\ref{zero-char}) and (\ref{nonzero-char}).

  Since we can do this for any such differential polynomials $p$ and $q$, we have that $\Kden/\mfr$ is a differentially closed field.
\end{proof}

In Appendix~\ref{ap:saturated}, we will sharpen \cref{thm:models-DCF} in the following way.

\begin{thm}\label[thm]{thm:main}
  For any maximal ideal $\mfr$ of $\Kden$, $\Kden/\mfr$ is a saturated model of $\DCF$ with $|\Kden/\mfr| = 2^{\aleph_0}$. In particular, the isomorphism type of $\Kden/\mfr$ does not depend on the choice of $\mfr$.
\end{thm}

\section{Hay in a haystack: Forcing, sheaves, and the Baire category theorem}
\label{sec:forcing}

Now we'll discuss how to connect the construction in Sections~\ref{sec:main} and \ref{sec:models} and the intuitive idea described in the introduction. We'll start by briefly sketching a proof of Seidenberg's embedding theorem from \cref{lem:Seidenberg-lemma}. (See Figure~\ref{fig:Seidenberg-embedding}.)
\begin{proof}
Think of a countable differential field $L$ as a tower of finitely generated differential field extensions $L_0 \subseteq L_1 \subseteq L_2 \subseteq \dots \subseteq L$. Then build a coherent family of embeddings of each $L_i$ into $\Kc(U_i)$ for a descending sequence of open sets $U_0 \supseteq U_1 \supseteq U_2 \supseteq \dots$ using \cref{lem:Seidenberg-lemma}. Choose the $U_i$'s in such a way as to ensure that $\bigcap_{i < \omega} U_i$ is a singleton $\{a\}$. This construction then yields an embedding of $L$ into the ring of germs of meromorphic functions at this $a$.
\end{proof}
Acolytes of the Baire category theorem will undoubtedly recognize this as bearing a marked similarity to the proof of their favorite theorem. And, indeed, this is no accident.

  \begin{figure}
  \centering
  \begin{tikzpicture}

    \draw (-2.5,-0.25) node {
      \begin{tikzcd}
        L & {\mathcal{K}_a} \\
	{\phantom{a}} & {} \\
        {L_2} & {\mathcal{K}(U_2)} \\
        {L_1} & {\mathcal{K}(U_1)} \\
        {L_0} & {\mathcal{K}(U_0)}
        \arrow[from=1-1, to=1-2]
        \arrow[hook, from=2-1, to=1-1]
        \arrow[from=2-2, to=1-2]
        \arrow["\cdots"{description,sloped}, draw=none, from=3-1, to=1-1]
        \arrow[hook, between={0}{0.9}, from=3-1, to=2-1]
        \arrow[from=3-1, to=3-2]
        \arrow["\cdots"{description,sloped}, draw=none, from=3-2, to=1-2]
        \arrow[between={0}{0.9}, from=3-2, to=2-2]
        \arrow[hook, from=4-1, to=3-1]
        \arrow[from=4-1, to=4-2]
        \arrow[from=4-2, to=3-2]
        \arrow[hook, from=5-1, to=4-1]
        \arrow[from=5-1, to=5-2]
        \arrow[from=5-2, to=4-2]
      \end{tikzcd}
    };

    \begin{scope}[shift={(1,0.5)},scale=2.5]
      \draw plot [smooth cycle,tension=1.1] coordinates {(-0.2,0) (1,1.3) (2,0) (1,-0.5) (0,-1.5)};
      \draw (0.1,-1.2) node {$U_0$};
      \draw  plot [smooth cycle,tension=1.1] coordinates {(0.3,0.5) (1,1) (1.7,0.3) (0.9,-0.3) (0.2,-1) (0,-0.5)};
      \draw (0.25,-0.73) node {$U_1$};
      \draw  plot [smooth cycle,tension=1.1] coordinates {(0.5,0.5) (1.1,0.9) (1.3,0.3) (0.6,-0.3) (0.2,-0.4)};

      \coordinate (B) at (0.35,-0.3);
      \draw (B) node {$U_2$};

      \coordinate (A) at (1,0.65);

      \foreach \i in {3,4,5}{
        \fill ($ (B)!\i/16!(A) $) circle (0.25pt);
      }

      \foreach \i [evaluate=\i as \t using \i/18] in {1,2,...,9,14}{
        \draw [opacity={\t^(1/2)}] plot [smooth cycle,tension=1.1] coordinates {($ (A)!\t^(1.1)!(0.5,0.5) $) ($ (A)!\t!(1.1,0.9) $) ($ (A)!\t^(0.9)!(1.3,0.3) $) ($ (A)!\t!(0.6,-0.3) $) ($ (A)!\t^(0.7)!(0.2,-0.2) $)};
      }

      \fill (A) circle (0.5pt) node[below] {$a$};

    \end{scope}
  \end{tikzpicture}
  \caption{The proof of Seidenberg's embedding theorem.}
  \label{fig:Seidenberg-embedding}
\end{figure}
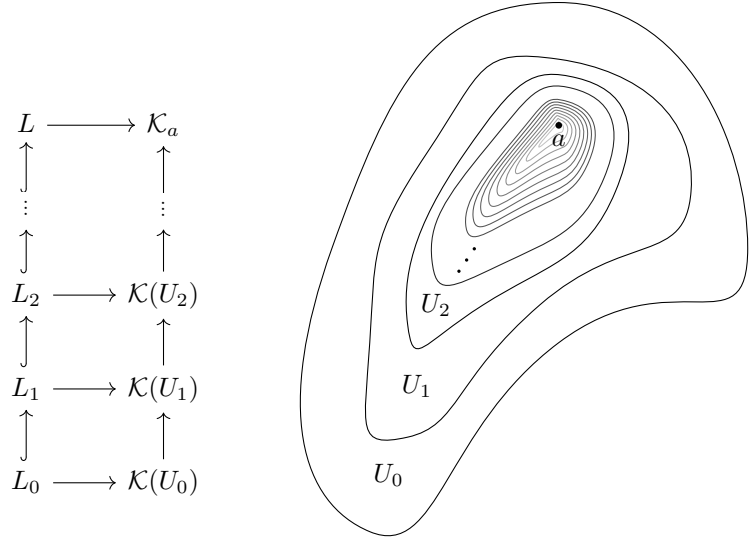

\subsection{Reproving Seidenberg's theorem with Baire's}
\label{sec:reproving-with-Baire}

For the sake of illustrating the relationship between Seidenberg's embedding theorem and the construction of $\Kden/\mfr$ as methods of building differentially closed fields, let's refactor the proof of Seidenberg's theorem using the following notions.

\begin{defn}\label[defn]{defn:Seidenberg-sheaf}$ $
  \begin{enumerate}
  \item\label{Seidenberg-sheaf} A \emph{Seidenberg sheaf} is a subsheaf $\Fc \subseteq \Kc$ of differential rings satisfying that for any non-empty connected open subset $U$ of $\Cb$, $\Fc(U)$ is a field and for any finitely generated differential ring $K \subseteq \Fc(U)$ and finitely generated abstract differential field extension $L \supseteq K$, there is a non-empty connected open subset $W \subseteq U$ and a differential field embedding $f : L \to \Fc(W)$ such that the following diagram commutes:
    \[\begin{tikzcd}
        L & {\Fc(W)} \\
        K & {\Fc(U)}
        \arrow["f", from=1-1, to=1-2]
        \arrow[hook, from=2-1, to=1-1]
        \arrow[hook, from=2-1, to=2-2]
        \arrow["{g\mapsto g|_W }"'{pos=0.45}, from=2-2, to=1-2]
      \end{tikzcd}\]
  \item\label{F-generic} Given a Seidenberg sheaf $\Fc$, a complex number $a$ is \emph{$\Fc$-generic} if for any connected open neighborhood $U \ni a$ and $K \subseteq \Fc(U)$ and $L \supseteq K$ as in the previous bullet point, we can find $W\subseteq U$ and $f : L \to \Fc(W)$ as before but additionally satisfying that $a \in W$.
  \item\label{locally-countable} A sheaf $\Fc \subseteq \Kc$ is \emph{locally countable} if for every $z$ and neighborhood $U \ni z$, there is a $V \subseteq U$ with $z \in V$ such that $|\Fc(V)| \leq \aleph_0$.\footnote{A more robust general definition of local countability would be to say that $\Fc$ is the sheafification of a presheaf $\Fc^0 \subseteq \Fc$ satisfying that $\Fc^0(U)$ is countable for every open set $U$, but in the context we care about this is equivalent to the stated definition.}
  \end{enumerate}
\end{defn}

So in other words, \cref{defn:Seidenberg-sheaf} part (\ref{Seidenberg-sheaf}) just takes the conclusion of \cref{lem:Seidenberg-lemma} and names it as a property. Note that if $\Fc$ is a sheaf of meromorphic functions, then $\Fc$ is locally countable if and only if $|\Fc(U)|\leq\aleph_0$ for every connected $U$.

We will use the following observation a few times.

\begin{lem}\label[lem]{lem:Seidenberg-basic-test}
  Fix a basis $B$ of connected open subsets of $\Cb$. If $\Fc \subseteq \Kc$ is a subsheaf satisfying \cref{defn:Seidenberg-sheaf} part (\ref{Seidenberg-sheaf}) for open sets $U$ in $B$, then $\Fc$ is a Seidenberg sheaf.
\end{lem}
\begin{proof}
  The condition that $\Fc(U)$ be a field for every connected open $U$ follows from gluing and the fact that division is computed pointwise.

  For the second condition, fix $U$, $K \subseteq \Fc(U)$, and $L \supseteq K$ as in \cref{defn:Seidenberg-sheaf} part (\ref{Seidenberg-sheaf}). Find $U' \subseteq U$ with $U' \in B$. The image $K'$ of $K$ under the restriction map $g \mapsto g|_{U'}$ is isomorphic to $K$, so we can find a field extension $L' \supseteq K'$ such that $L / K$ and $L' / K'$ are isomorphic field extensions. Applying the assumed condition to $U'$, $K' \subseteq \Fc(U')$, and $L' \supseteq K'$ gives an open set $W \subseteq U'$ and a field embedding $f : L' \to \Fc(W)$. Composing $f$ with the isomorphism $L\cong L'$ now gives the required map. (See Figure~\ref{fig:diag-basic-test}.)
\end{proof}

  \begin{figure}
  \centering
  \begin{tikzpicture}

    \draw (-3.25,-0.25) node {\begin{tikzcd}
        L & {L'} & {\mathcal{K}(W)} \\
        & {K'} & {\mathcal{K}(U')} \\
        K && {\mathcal{K}(U)}
        \arrow["\cong", from=1-1, to=1-2]
        \arrow["f", from=1-2, to=1-3]
        \arrow[hook, from=2-2, to=1-2]
        \arrow[hook, from=2-2, to=2-3]
        \arrow["{g\mapsto g|_{W}}"', from=2-3, to=1-3]
        \arrow[hook, from=3-1, to=1-1]
	\arrow["\cong"{sloped}, from=3-1, to=2-2]
        \arrow[hook, from=3-1, to=3-3]
        \arrow["{g\mapsto g|_{U'}}"', from=3-3, to=2-3]
      \end{tikzcd}};

    \begin{scope}[shift={(1,0)},scale=1.5]
      \draw plot [smooth cycle,tension=1.1] coordinates {(-0.2,0) (1,1) (2,0) (1,-0.5) (0,-1.5)};
      \draw (0,-0.85) node {$U$};
      \draw [dotted] plot [smooth cycle,tension=1.1] coordinates {(0.3,0.5) (1.25,0.8) (1.7,0.3) (1.2,-0.3) (0.2,-0.5)};
      \draw (0.3,-0.15) node {$U'$};
      \draw [dashed,fill=black,fill opacity=0.1] plot [smooth cycle,tension=1.1] coordinates {(0.4,0.3) (1.25,0.7) (1.2,0.3) (1.2,-0.15)};
      \draw (1,0.25) node {$W$};
    \end{scope}
  \end{tikzpicture}
  \caption{The construction in \cref{lem:Seidenberg-basic-test}.}
  \label{fig:diag-basic-test}
\end{figure}

We can now give our refactored proof of Seidenberg's embedding theorem. Recall that given a sheaf $\Fc$, we write $\Fc_a$ for the stalk of $\Fc$ at $a$ (i.e., the collection of germs of sheaves at $a$).

\begin{prop}\label[prop]{prop:Seidenberg-refactor}
  Let $\Fc \subseteq \Kc$ be a subsheaf of meromorphic functions.
  \begin{enumerate}
  \item\label{countable-closure} If $\Fc$ is locally countable, then there is a locally countable Seidenberg sheaf $\Fc' \subseteq \Kc$ with $\Fc \subseteq \Fc'$.
  \item\label{Seidenberg-stalk} If $\Fc$ is a Seidenberg sheaf, then for any $\Fc$-generic $a$, $\Fc_a$ is a universal differential field (and so is differentially closed).
  \item\label{Seidenberg-generic} If $\Fc$ is a locally countable Seidenberg sheaf, then the set of $\Fc$-generic complex numbers is comeager in $\Cb$ (and so in particular is non-empty). Moreover, for any two $\Fc$-generic complex numbers $a$ and $b$, $\Fc_a$ and $\Fc_b$ are isomorphic as differential fields.
  \end{enumerate}
\end{prop}
\begin{proof}
  Fix a countable basis $(U_i)_{i < \omega}$ of connected open subsets of $\Cb$.

  (\ref{countable-closure})  Let $\Fc^0 = \Fc$. Given $\Fc^{n}$, construct $\Fc^{n+1}$ as follows: For each $i < \omega$, finitely generated differential field $K \subseteq \Fc(U_i) \subseteq \Kc(U_i)$ and finitely generated abstract differential field extension $L \supseteq K$ (up to isomorphism of differential field extensions of $K$), find $W_{i,K,L} \subseteq U_i$ and a differential field embedding $f_{i,K,L} : L \to \Kc(W_{i,K,L})$ as in \cref{lem:Seidenberg-lemma}. Let $\Fc^{n+1}$ be the smallest subsheaf of $\Kc$ containing $\Fc^n$ and the sections $f_{i,K,L}(L) \subseteq \Kc(W_{i,K,L})$ and satisfying that for any non-empty connected open set $U$, $\Fc^{n+1}(U)$ is a field. Since \DCF{} is $\aleph_0$-stable, there are only countably many triples $(i,K,L)$ (since we are considering $L$ up to isomorphism), so it easily follows that $\Fc^{n+1}$ is locally countable. Finally, let $\Fc'$ be the sheafification of the presheaf $\Fc^\omega(U) \coloneq \bigcup_{n < \omega}\Fc^n(U)$ (i.e., $\Fc'$ is the smallest subsheaf of $\Kc$ satisfying $\Fc^\omega(U) \subseteq \Fc'(U)$ for each open $U$). Again, $\Fc'$ is easily seen to be locally countable. It is now immediate that $\Fc'$ satisfies the Seidenberg sheaf condition for open sets of the form $U_i$, so $\Fc'$ is a Seidenberg sheaf by \cref{lem:Seidenberg-basic-test}.

  (\ref{Seidenberg-stalk}) is an exercise in definitions.

  (\ref{Seidenberg-generic}) By the same argument as in \cref{lem:Seidenberg-dense}, $\Fc$ satisfies that for any non-empty open $U$, finitely generated differential field $K \subseteq \Fc(U)$, and finitely generated abstract differential field extension $L \supseteq K$, there is an open $W \subseteq U$ dense in $U$ (with $W$ not necessarily connected) and a differential ring embedding $f : L \to \Fc(W)$ (which restricts to a field embedding on each connected component of $W$) such that the diagram in \cref{defn:Seidenberg-sheaf} part (\ref{Seidenberg-sheaf}) commutes.

  For each $i$ and choice of $K \subseteq \Fc(U_i)$ and $L \supseteq K$ (again up to isomorphism) as in the previous paragraph, fix a corresponding choice of open $W_{i,K,L}\subseteq U_i$ dense in $U_i$. Let $V_{i,K,L} = W_{i,K,L} \cup \tint (\Cb \setminus U_i)$. Clearly $V_{i,K,L}$ is a dense open set. There are (again) only countably many triple $(i,K,L)$. Therefore $X \coloneq \bigcap_{i,K,L} V_{i,K,L}$ is comeager and is non-empty by the Baire category theorem. It is now easy to verify that any element of $X$ is $\Fc$-generic (again using \cref{lem:Seidenberg-basic-test}).

  Finally, to see that for any $\Fc$-generic $a$ and $b$, $\Fc_a$ is isomorphic to $\Fc_b$, note that $\Fc_a$ and $\Fc_b$ are both countable $\aleph_0$-saturated models of \DCF{} and are therefore elementarily equivalent saturated structures of the same cardinality.
\end{proof}

So by \cref{prop:Seidenberg-refactor} every locally countable Seidenberg sheaf gives us not a single differentially closed field but a whole family of them, indexed by some comeager set of points in $\Cb$. If we had a `natural' example of such a sheaf, we could argue for a similar level of naturality to something like the probabilistic construction of the random graph.

There actually are fairly easy to describe locally countable Seidenberg sheaves, depending on who you're talking to. For example, say that a meromorphic function $f$ is \emph{locally computable} if for every rational $a$ in the domain of $f$, the Taylor series of $f$ at $a$ is a uniformly computable (in the sense of computable analysis) sequence of complex numbers. It is almost immediate that the collection of such functions forms a locally countable sheaf on $\Cb$. It is also tedious, but not fundamentally difficult, to verify that this is a Seidenberg sheaf. To a computability theorist, this is a perfectly natural sheaf to consider, and unlike the sheaves built with \cref{prop:Seidenberg-refactor} part (\ref{countable-closure}), it requires no arbitrary choices.

Nevertheless, there is something unsatisfying about needing to specify a proper subsheaf of $\Kc$ at all. $\Kc$ itself is already a Seidenberg sheaf; that's why we're talking about the concept in the first place. Wouldn't it be better if we could just use $\Kc$, the most canonical choice of subsheaf of $\Kc$, as our sheaf directly?

But now we've arrived at our childish idea again. There are no $\Kc$-generic complex numbers. This is a manifestation of the hay-in-a-haystack phenomenon: If we squint, we can see a uniform generic behavior in the space $\Cb$, a behavior that almost all\footnote{In the sense of the Baire category theorem rather than in the sense of measure theory.} points in $\Cb$ `should' exhibit. So all we need to do is grab a stalk of hay out of the haystack. This should be easy, given that `most' of the haystack is hay, but if we try to reach in and grab a specific complex number $a$, we always find a non-generic needle, as witnessed by $\frac{1}{z-a}$.

This seems like an impossible puzzle. It would be enough to have a complex number that sits inside every dense open subset of $\Cb$, but this excludes every complex number that exists in the mathematical universe. Let's indulge for a moment in even more childishness and fantasy. What if there were other, secret complex numbers, nestled between the familiar complex numbers we know and love, yet hidden in the shadows outside the universe? Such a complex number would surely not be the pole of any meromorphic function in our universe, since a pole has to be a pole \emph{somewhere}. Moreover, it could easily be the case that some such number sits inside all dense open subsets of $\Cb$ in our universe (or, more specifically, all dense open sets that can be described by data in our universe). Assuming we take $\Kc$ to be the sheaf of meromorphic functions from our universe, such a number $a$ would be $\Kc$-generic and by \cref{prop:Seidenberg-refactor} part (\ref{Seidenberg-stalk}), the stalk $\Kc_a$ would be differentially closed.

Colorful cosmic Platonism aside, this is essentially the idea behind set-theoretic forcing, or at least behind Cohen forcing, the original kind of forcing used by Cohen to establish the independence of the continuum hypothesis.

\subsection{Re-reproving Seidenberg's theorem with Cohen's}
\label{sec:re-reproving}

\epigraph{When the manuscript with Cohen's proof of the independence of the Continuum Hypothesis came to Wroc\l{}aw, Ryll-Nardzewski took it home for the evening. Next day he came to the Institute and commented ``Well, it is just the Baire theorem.''}{Wroc\l{}aw urban legend \cite[Sec.~5.2]{borodulin2023wroclaw}}

Cohen's original approach to forcing---the approach of using countable transitive models---proceeds by appealing to the L\"owenheim-Skolem theorem (or in other words to Skolem's paradox), which essentially amounts to a mathematical conspiracy akin to the 1998 film \emph{The Truman Show}. In the movie, the main character's whole life is, unbeknownst to him, merely a television show being filmed on a giant island-sized set. Similarly, a countable transitive model $M$ appears from the inside to be a perfectly ordinary mathematical universe, containing objects that, for instance, seem uncountable, but secretly $M$ is contained in some much larger universe in which $M$ is merely a countable set.

For a countable transitive model $M$, the set $\Cb \cap M$ of complex numbers in $M$ is a proper subset of the true set of complex numbers $\Cb$. We would like to be able to talk about the sheaf $\Kc^M$ of meromorphic functions that $M$ is `able to see' as a subsheaf of $\Kc$, the actual sheaf of meromorphic functions. The issue is that meromorphic functions in $M$ aren't literally meromorphic functions, because they are only functions on $\Cb\cap M$, not on $\Cb$. This is fairly easy to deal with, since objects that $M$ identifies as meromorphic functions are still uniformly continuous and admit unique extensions to actual meromorphic functions. When talking about sheaves, there's an additional wrinkle which is that a given open set might not be `visible' to $M$ either. This is why the following definition is maybe a little bit less straightforward than it seems like it should be.

\begin{defn}\label[defn]{defn:M-meromorphic-sheaf}
  Given a transitive model $M$ of \ZFC{} and an open set $U$, $\Kc^M(U)$ is the differential ring of meromorphic functions $f$ on $U$ satisfying that for each $z \in U$, there is an open $V \ni z$ with $V \subseteq U$ such that $(f|_V) \cap M \in M$.\footnote{This can be characterized similarly to local computability. In particular, $\Kc^M(U)$ is the set of $f \in \Kc(U)$ satisfying that for each rational $z$ in the domain of $f$, the Taylor series of $f$ at $z$ is an element of $M$.}
\end{defn}

It is immediate\footnote{One might worry that the extensions of what $M$ thinks are meromorphic functions to actual functions might not be meromorphic, but this isn't an issue because $M$ can compute derivatives correctly. In general, though, one of the biggest difficulties with becoming comfortable with forcing and other ideas from set theory that involve passing between different models is getting a handle on what facts are \emph{absolute} (i.e., remain true when passing to and from $M$) and what facts are not.} that $\Kc^M$ is a subsheaf of $\Kc$. $\Kc^M$ is isomorphic (in a sense that is slightly obnoxious to specify carefully) to $M$'s internal sheaf of meromorphic functions (i.e., the object you get if you build the sheaf of meromorphic functions while trapped inside $M$).

Forcing is all about `adding new elements' to a given model of \ZFC. Given a countable transitive model $M$ and some object $G$ not in $M$, we'd like to be able to talk about the model $M[G]$ of \ZFC{} `generated by' $M\cup \{G\}$. This doesn't always make sense, however, and the easiest way to ensure that it does is for $G$ to be in some sense `generic relative to $M$.'

Often (and in the case we care about here), forcing is specifically about adding new real numbers (or rather complex numbers for us, although there isn't really a difference set-theoretically). The original notion of genericity in this context is the following.

\begin{defn}\label[defn]{defn:Cohen-generic}
  Fix a transitive model $M$ of \ZFC.
  \begin{enumerate}
  \item Given $U_0 \subseteq \Cb \cap M$ with $U_0 \in M$, we say that $U_0$ \emph{codes the open set $U$} if for every $z\in \Cb$, $z \in U$ if and only if there is an open neighborhood $V \ni z$ such that $V \cap M \subseteq U_0$.
  \item An open set $U \subseteq \Cb$ is \emph{coded in $M$} if it is coded by some $U_0 \in M$.
  \item A complex number $a$ is \emph{Cohen generic over $M$} if for every dense open $U \subseteq \Cb$ coded in $M$, $a \in U$.
  \end{enumerate}
\end{defn}

Transitive models aren't required to be countable, and it can be the case that Cohen generics exist over uncountable transitive models, but \cref{defn:Cohen-generic} makes it clear what Ryll-Nardzewski's point was. If $M$ is a countable transitive model, then there are complex numbers Cohen generic over $M$ just by the Baire category theorem.

\begin{prop}\label[prop]{prop:ZFC-Seidenberg}
  Fix a transitive model $M$ of \ZFC.
  \begin{enumerate}
  \item\label{ctm-Seidenberg-sheaf} $\Kc^M$ is a Seidenberg sheaf.
  \item\label{DCF-Cohen-stalk} If $a \in \Cb$ is Cohen generic over $M$, then it is $\Kc^M$-generic.
  \end{enumerate}
\end{prop}
\begin{proof}
  (\ref{ctm-Seidenberg-sheaf}) Since $M$ is a model of (enough of) \ZFC, it satisfies \cref{lem:Seidenberg-lemma} internally. This (together with the fact that \DCF{} is $\aleph_0$-stable, so every isomorphism type of finitely generated characteristic $0$ differential fields is represented in $M$) implies that $\Kc^M$ satisfies \cref{defn:Seidenberg-sheaf} part (\ref{Seidenberg-sheaf}) for open sets $U \subseteq \Cb$ coded in $M$. Therefore by \cref{lem:Seidenberg-basic-test}, $\Kc^M$ is a Seidenberg sheaf.

  (\ref{DCF-Cohen-stalk}) This proceeds similarly to the proof of \cref{prop:Seidenberg-refactor} part (\ref{Seidenberg-generic}) (and so also similarly to the proof of \cref{lem:Seidenberg-dense}). Work internally in $M$. For any non-empty connected open set $U \subseteq \Cb$, finitely generated differential field $K \subseteq \Kc(U)$, and finitely generated differential field extension $L \supseteq K$, there is an open $W \subseteq U$ dense in $U$ (with $W$ not necessarily connected) and a differential ring embedding $f : L \to \Kc(W)$ such that the diagram in Figure~\ref{fig:Seidenberg-1} commutes.

  Now suppose that $a$ is Cohen generic over $M$. For any such $W \in M$ coding an open set $W' \subseteq \Cb$, we must have that $a \in W' \cup \tint(\Cb \setminus W')$. Since we can do this for any such $W$, we have that $a$ is $\Kc^M$-generic, as required.
\end{proof}

Given any countable transitive model $M$, there are complex numbers Cohen generic over $M$, so for any such $a \in \Cb$, the stalk $\Kc^M_a$ is a universal differential field. In particular, this is a substalk of $\Kc_a$. Since every countable characteristic $0$ differential field embeds into a universal differential field, we have a third proof\footnote{Strictly speaking \ZFC{} itself doesn't prove that transitive models of \ZFC{} exist, but like almost all size issues, this can be dealt with a number of ways. One way is to replace \ZFC{} with a `large enough' fragment of \ZFC. Another way is to just assume the existence of some very mild large cardinals.} of Seidenberg's embedding theorem.

Now, one may object that all we've done is rephrased our second proof, the argument in \cref{sec:reproving-with-Baire} (itself a rephrasing of the original proof), in terms of set-theoretic forcing, but as Ryll-Nardzewski reportedly commented, forcing is in some deep sense `just' the Baire category theorem. That said, I also do think there's value in this perspective (otherwise I wouldn't be writing about it). Once you become comfortable with it, it can be a fruitful organizational framework, at least for seeing that something should be true, such as the proof of \cref{fact:perf-soln} alluded to in the appendix.

\subsection{Boolean-valued models and $\neg\neg$-sheaves}
\label{sec:Boolean-valued}

The approach to forcing in terms of countable transitive models has a similar limitation to that of \cref{prop:Seidenberg-refactor}. Specifically, when we want to reason about meromorphic functions, we need to pick some specific at most countable set of them to embed into a differentially closed field. Likewise, when we build our little fake universe $M$, we have to decide which countable bundle of data to include. It is more convenient (and, hence, often done by set theorists) to pretend that we already live in the fake universe. In terms of our earlier metaphor, this is like thinking of the whole world as a stage, rather than building Truman's island-sized set. This is the same reversal that is suggested in general by simulated reality thought experiments, such as \emph{The Truman Show} or Plato's allegory of the cave. Contemplating a miniature simulacrum of the entire world invites the question of whether the world we live in might likewise be a fa\c{c}ade. (See Figure~\ref{fig:Flammarion}.)

\begin{figure}
  \centering
  \includegraphics[width=\textwidth]{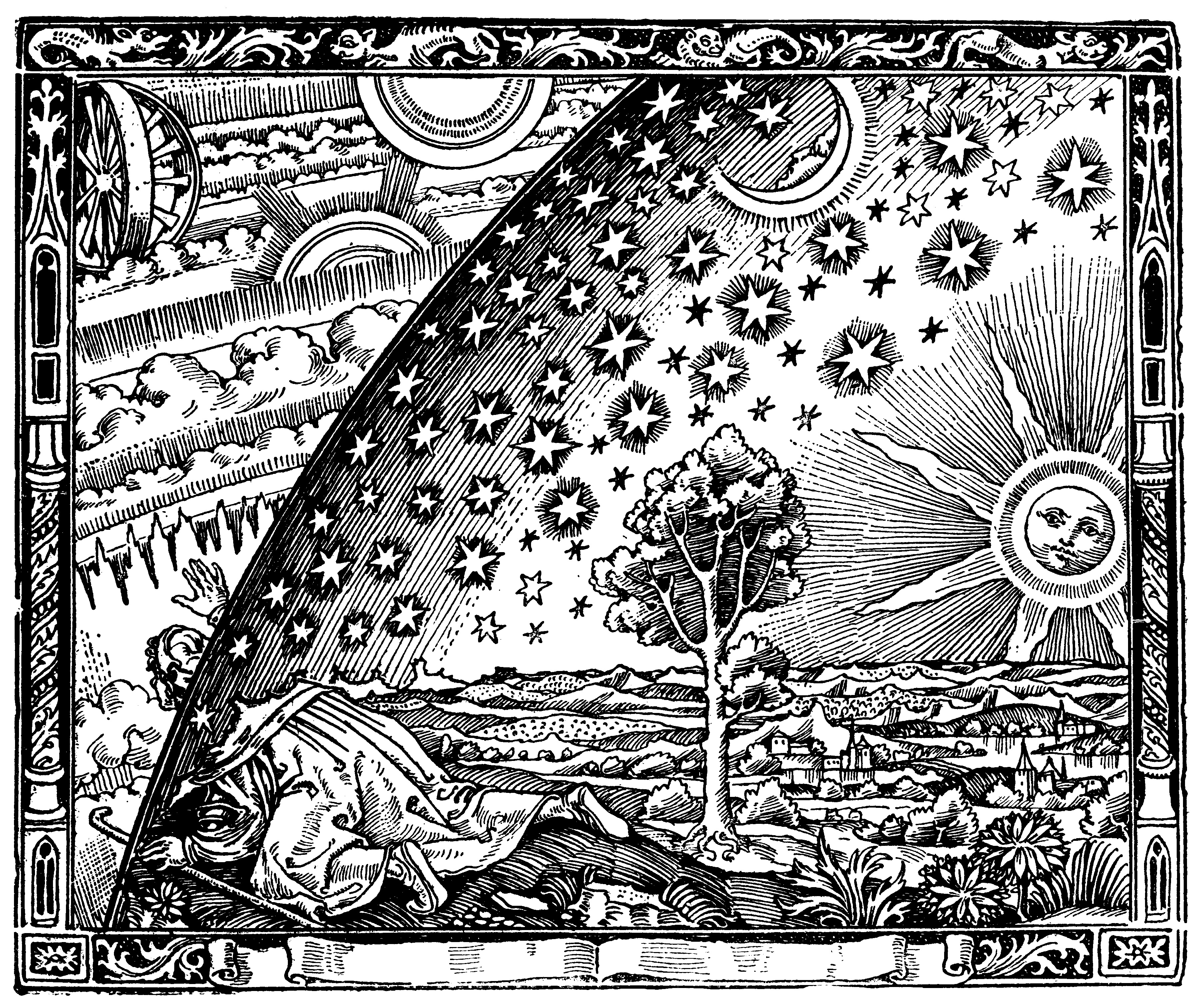}
  \caption{Unknown artist. \emph{Flammarion engraving}. Circa 1500-1800. Wood engraving.}
  \label{fig:Flammarion}
\end{figure}

Making this approach to forcing rigorous can be done in a couple of ways, one of which is the formalism of Boolean-valued models. Boolean-valued models make sense for arbitrary first-order theories, not just set theory, but they're primarily discussed in the context of forcing. The idea is that we replace the ordinary set of two truth values with some arbitrary complete Boolean algebra $B$. What this allows us to do is talk about what properties a, say, Cohen generic complex number \emph{would} have if it existed, regardless of whether it actually exists or not. Crucially though, this can all take place `inside the actual universe.'

The full formalism of Boolean-valued models of set theory is a bit much to get into here, but the basic idea is that one considers a hierarchy of sets $V^B$ (defined recursively in the same way as von Neumann's cumulative hierarchy $V$) where set membership takes values in the complete Boolean algebra $B$, rather than the familiar $\{\text{true},\text{false}\}$.

We can elaborate on this idea in the context of $\Kden$ specifically. Let $\Lc(\Kden)$ be the language of differential rings extended with constants for each element of $\Kden$. Let $\RO(\Cb)$ be the (complete) Boolean algebra of regular open subsets of $\Cb$. We should pause to note that $\RO(\Cb)$ is precisely the quotient of the lattice of opens in $\Cb$ by the equivalence relation of agreeing on a dense open set (i.e., $U \sim V$ if there is a dense open set $W$ such that $U\cap W = V \cap W$), which perhaps bears a resemblance to \cref{defn:Kden-ring}.

Define a $\RO(\Cb)$-valued truth assignment $\nu$ on $\Lc(\Kden)$-sentences as follows:
\begin{itemize}
\item For any equation of terms $a = b$, we can interpret $a$ and $b$ as elements of $\Kden$ in the obvious way. Then $\nu(a = b)$ is $\tint(\cl\{z \in \Cb : a(z) = b(z)\})$.
\item $\nu(\varphi \wedge \psi) = \nu(\varphi) \cap \nu(\psi)$ and $\nu(\varphi \vee \psi) = \tint(\cl(\nu(\varphi) \cup \nu(\psi)))$.
\item $\nu(\neg\varphi) = \tint(\cl(\Cb \setminus \nu(\varphi)))$.
\item $\nu(\forall x \varphi(x)) = \bigcap_{a \in \Kden} \nu(\varphi(a))$ and $\nu(\exists x \varphi(x)) = \tint\left( \cl \bigcup_{a \in \Kden}\nu(\varphi(a)) \right) $.
\end{itemize}

Now by thinking about the proofs of the statements earlier in this paper, it's not too hard to convince yourself of the following proposition.

\begin{prop}\label[prop]{prop:Bool-valued-DCF}
  For every axiom $\varphi$ of \DCF{}, $\nu(\varphi) = \Cb$. In other words, the valuation $\nu$ makes $\Kden$ into a Boolean-valued model of \DCF.
\end{prop}

As a simple example of \cref{prop:Bool-valued-DCF}, consider the basic field axiom $\forall x (x = 0 \vee \exists y (xy = 1))$. Given $x = [f] \in \Kden$, we have two open sets, the set $U = \{z : f~\text{vanishes in some neighborhood of}~z\}$ and the set $V = \{z : f~\text{is nowhere locally vanishing in some neighborhood of}~z\}$. Locally on $V$, we can choose $y = [f]^\dagger$ as an inverse of $[f]$. We now have that $\nu([f] = 0) = \tint(\cl U)$ and $\nu([f][f]^\dagger = 1) = \nu(\exists y([f]y = 1)) = \tint(\cl V)$. Since the union of $U$ and $V$ is dense, this gives $\nu([f]= 0 \vee \exists y([f]y = 1)) = \Cb$. Since this is the case for every $[f] \in \Kden$, we have that $\nu(\forall x(x = 0 \vee \exists y(xy = 1))) = \Cb$.

We can connect this to the previous account of forcing in the following way. Fix a transitive model $M$ of \ZFC{}. Let
\[
\Kden^M = \{[f] \in \Kden : f \in \Kc^M(U),~U~\text{dense open coded in}~M\}.
\]
Note that $\Kden^M$ is (isomorphic to) $M$'s internal copy of $\Kden$ (i.e., the object you would get if you built $\Kden$ while trapped inside $M$). Since each $[f] \in \Kden^M$ has a representative $f$ whose domain is a dense open set coded in $M$, we have that for any Cohen generic $a$ over $M$, there is a surjective differential ring morphism from $\Kden^M$ to $\Kc^M_a$.\footnote{The astute reader will have noticed that this means that $\Kc^M_a \cong \Kden^M/\mfr$ for some maximal ideal $\mfr$ in $\Kden^M$.} This allows us to think of $\Kc^M_a$ as an $\Lc(\Kden^M)$-structure. Specifically, the constant $[f]$ is interpreted in $\Kc^M_a$ as the germ of $g$ at $a$ for any representative $g$ of $[f]$ with $a \in \dom(g)$.

With this interpretation, we get the following transfer theorem.

\begin{prop}
  Fix a transitive model $M$ and an $\Lc(\Kden^M)$-sentence $\varphi$.
  \begin{enumerate}
  \item For any complex number $a$ which is Cohen generic over $M$, $\Kc^M_a$ satisfies $\varphi$ if and only if $a \in \nu(\varphi)$.
  \item If Cohen generic complex numbers over $M$ exist, then
    \[
      \nu(\varphi) = \tint(\cl\{a \in \Cb : a~\text{is Cohen generic over}~M,~\Kc^M_a~\text{satisfies}~\varphi\}).
    \]
  \end{enumerate}
\end{prop}

This kind of reasoning can actually be used to prove \cref{prop:Bool-valued-DCF} as a corollary of \cref{prop:ZFC-Seidenberg}.

Forgetting about transitive models of \ZFC{} for a minute, the Boolean-valued perspective of $\Kden$ can arguably be used to give a visceral `explanation' of the relevance of the theory of differentially closed fields to meromorphic functions. Although we've phrased $\Kden$ as a ring all along, it's really more naturally thought of as a sheaf. Specifically, $\Kden(U)$ is the differential ring of meromorphic functions $f$ whose domains are dense open subsets of $U$, modulo agreement on dense open subsets (i.e., just take \cref{defn:Kden-ring} and relativize it to the open set $U$). $\Kc$ itself can obviously be thought of as a subsheaf of $\Kden$, making $\Kden$ into a canonical `Boolean-valued differential closure' of the sheaf $\Kc$ of meromorphic functions itself. Part of the reason we can get away with just thinking of $\Kden$ as a ring is that as a sheaf it is flabby (i.e., satisfies that every local section extends to a global section), so no data is lost if we only consider global sections.

$\Kden$ as a sheaf already has a name. It's the $\neg\neg$-sheafification of $\Kc$ (i.e., the sheafification of $\Kc$ with regards to the double negation Grothendieck topology in which we consider a family of open sets $(U_i)_{i \in I}$ to cover an open set $V$ if $\bigcup_{i \in I}U_i$ is dense in $V$). So we could have written $\Kden$ as $\Kc_{\neg\neg}$. This way of thinking about these objects is intimately tied to the sheaf- or topos-theoretic approach to forcing \cite{MacLane1994}. $\Kden$ is an object in $\Sh_{\neg\neg}(\Cb)$, the topos of $\neg\neg$-sheaves on $\Cb$, which is the minimal dense subtopos of $\Sh(\Cb)$, the ordinary topos of sheaves on $\Cb$. $\Sh_{\neg\neg}(\Cb)$ is a relatively easy example of a non-trivial topos with no points (and, relatedly, a non-trivial locale with no points).

The valuation $\nu$ we have defined here corresponds to truth values of statements about $\Kden$ as a differential ring in the (classical) internal logic of $\Sh_{\neg\neg}(\Cb)$. So what we can say is that $\Kden$ is internally a differentially closed field in $\Sh_{\neg\neg}(\Cb)$. Since $\Sh_{\neg\neg}(\Cb)$ corresponds to the Boolean-valued presentation $V^{\RO(\Cb)}$ of Cohen forcing, this fact is the same thing as our main idea: The stalk of the sheaf of meromorphic functions (from our universe) at a Cohen generic complex number is a differentially closed field.

\subsection{The stalk at a point that (doubly) doesn't exist}
\label{sec:doesn't-exist}

While the Boolean-valued or sheaf-theoretic understanding of the relationship between $\Kden$ and forcing is nice, some people\footnote{Including me, on some days.} might find it to be `too enlightened' for their tastes. Actual $\{\text{true},\text{false}\}$-valued models have a certain appealing conceptual simplicity, reminiscent of the bluntness of maximal ideals in rings, and I would argue this is big part of why Boolean-valued models are not popular as a framework for forcing among set theorists.

So how can we rectify these two goals? We want to build a differentially closed field out of the totality of the sheaf of meromorphic functions without compromising by passing to a subsheaf, but we also don't want to achieve our goal only in some generalized sense that requires changing our notion of what the words `field' or `truth' mean.

Putting on our model-theorist hats for a second, we could take another look at the set of `fully true' $\Lc(\Kden)$-sentences in $\Kden$ interpreted as a Boolean-valued structure,
\[
T_{\Kden} \coloneq \{\varphi \in \Lc(\Kden) : \nu(\varphi) = \Cb\}.
\]
We know by \cref{prop:Bool-valued-DCF} that this extends the theory of characteristic $0$ differentially closed fields, \DCF, but we actually have more. Recall the following concept from the typical proof of the completeness theorem for first-order logic.

\begin{defn}
  A first-order theory $T$ is \emph{term-complete} if for any formula $\varphi(x)$, there is a term $t$ such that $(\exists x \varphi(x)) \to \varphi(t)$ is in $T$.
\end{defn}

The way the completeness theorem is typically proven is by first arguing that any consistent theory $T$ can be extended to a consistent term-complete theory $T'$, then arguing that a model of $T$ can be read off from any completion $T''$ of $T'$ (specifically by considering the term structure of $T''$).

\begin{prop}\label[prop]{prop:term-complete}
  $T_{\Kden}$ is term-complete.
\end{prop}
\begin{proof}
  Fix an $\Lc(\Kden)$-formula $\varphi(x)$. Let $U = \nu(\exists x \varphi(x))$. By definition, we can find a family $(V_i)_{i \in I}$ of pairwise disjoint open subsets of $U$ and elements $([f_i])_{i \in I}$ of $\Kden$ such that $\nu(\varphi([f_i])) \supseteq V_i$ for each $i \in I$ and $\bigcup_{i \in I}V_i$ is dense in $U$. Let
  \[
    g(z) =
    \begin{cases}
      f_i(z) & z \in V_i \cap \dom(f_i) \\
      0 & z \in \tint(\Cb \setminus U)
    \end{cases}.
  \]
  We now have that $\nu(\varphi([g])) = \nu(\exists x \varphi(x))$, whereby $\nu((\exists x \varphi(x)) \to \varphi([g])) = \Cb$.
\end{proof}

In some perhaps unsurprising sense what we're doing here is sheaf-theoretic gluing. The topos $\Sh_{\neg\neg}(\Cb)$ is totally disconnected, so we can patch any family of local solutions to the `problem' of realizing $\varphi(x)$ together to form a global solution. We actually were doing this already back in \cref{sec:main}. The weak inverse $[f]^\dagger$ of a given $[f] \in \Kden$ is such a witness for the sentence $\exists x(x[f] = 1)$.

This now gives us a family of models of \DCF. For any completion $T'$ of $T_{\Kden}$, we can consider the term structure of $T'$: Elements are equivalence classes $[[f]]$ of the equivalence relation $[g]\sim [h]$, which holds if $[g] = [h]$ (or, equivalently, $[g]-[h] = 0$) is a sentence in $T'$. This equivalence relation respects the differential ring operations (because $T'$ extends the theory of differential rings), so we can interpret the differential ring operations on this structure, and by basic reasoning from the proof of the completeness theorem, we have that this thing is a model of $T' \supseteq \DCF$ and therefore a differentially closed field.

But in truth we've seen this object already. Given a completion $T'$ of $T_{\Kden}$, consider the ideal
\[
  \mfr_{T'} = \{[f]\in\Kden : [f] = 0~\text{is in}~T'\}.
\]
The term structure of $T'$ is precisely $\Kden/\mfr_{T'}$. Moreover, every maximal ideal of $\Kden$ is of this form, and a little bit more can be said.

\begin{prop}\label[prop]{prop:one-to-one-to-one}$ $
  There is a canonical one-to-one-to-one correspondence between completions of $T_{\Kden}$, ultrafilters on $\RO(\Cb)$, and maximal ideals in $\Kden$ given by the following:
    \begin{enumerate}
    \item For a completion $T' \supseteq T_{\Kden}$, the set $\{U \in \RO(\Cb) : [\onebf_U] = 1~\textnormal{is in}~T'\}$ is an ultrafilter on $\RO(\Cb)$.
    \item For an ultrafilter $F$ on $\RO(\Cb)$, $\{ 1-[\onebf_U] : U \in F \}$ generates a maximal ideal on $\Kden$.
    \item For a maximal ideal $\mfr$ on $\Kden$, $T_{\Kden} \cup \{[\onebf_U] = 0 : [\onebf_U] \in \mfr\}$ axiomatizes a completion of $T_{\Kden}$.
    \end{enumerate}
\end{prop}

Moreover, the Stone space of completions of $T'$, the Stone space of $\RO(\Cb)$, and the spectrum of $\Kden$ can be canonically identified as topological spaces.

The relevance of ultrafilters on the Boolean algebra $\RO(\Cb)$ is that the construction we're considering here can also be thought of as an example of the Boolean ultrapower construction introduced by Mansfield in \cite{Mansfield1971}.

So we really have returned to the construction from Sections~\ref{sec:main} and \ref{sec:models}, but what does it have to do with forcing? Looking at \cref{prop:term-complete}, one might be struck by how soft the proof is. Really, we didn't use any properties of $\Kden$ at all. You can see this too in \cref{prop:one-to-one-to-one}. Ultrafilters on $\RO(\Cb)$ don't have anything in particular to do with $\Kden$ or $T_{\Kden}$ or even differentially closed fields at all; rather, they `belong' to the space $\Cb$ itself.

Boolean ultrapowers are central in a third approach to forcing, referred to as the \emph{filter-quotient construction} in topos theory \cite[Sec.~V.9]{MacLane1994} and introduced in set-theoretic language by Hamkins and Seabold as the \emph{naturalist account of forcing} in \cite{HamkinsSeabold2012}. In Boolean-valued forcing, we talk about properties the expanded universe \emph{would} have if it existed, without actually committing to specifics (except perhaps in some `dynamic' sense). In the Boolean ultrapower approach, we need to actually build two universes on top of the one we start with.

The issue is that a Cohen generic complex number is something that needs to `live in the cracks' between the complex numbers we're able to see. But if we're in the `real' mathematical universe, then there are no such cracks, period. The complex numbers, like the reals, are complete. So we first need to build more `ordinary' complex numbers so that there will then be cracks in which Cohen generic complex numbers can live. Unfortunately this muddles our \emph{Truman Show}/allegory of the cave metaphor a bit, but to some extent this is because the semantics of the different approaches to forcing, while closely related, are in fact distinct. The issue is that there are now two distinct notions of `passing to a larger universe,' when before there was only one. (See Figure~\ref{fig:three-approaches}.)

  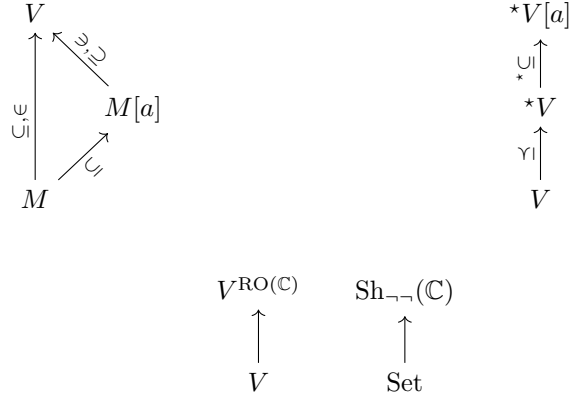
\begin{figure}
  \centering
  \[\begin{tikzcd}[column sep=small]
	V &&&& {^\star V[a]} \\
	& {M[a]} &&& {^\star V} \\
	M &&&& V \\
	&& {V^{\mathrm{RO}(\mathbb{C})}} & {\mathrm{Sh}_{\neg\neg}(\mathbb{C})} \\
	&& V & {\mathrm{Set}}
	\arrow["{\ni,\supseteq}"{sloped}', from=2-2, to=1-1]
	\arrow["{^\star\subseteq}"{sloped,pos=0.4}, from=2-5, to=1-5]
	\arrow["{\subseteq,\in}"{sloped,pos=0.4}, from=3-1, to=1-1]
	\arrow["\subseteq"'{sloped}, from=3-1, to=2-2]
	\arrow["\preceq"{sloped}, from=3-5, to=2-5]
	\arrow[from=5-3, to=4-3]
	\arrow[from=5-4, to=4-4]
\end{tikzcd}\]
  \caption{Three semantics approaches to forcing. The countable transitive model approach involves two arbitrary choices (the model $M$ and the generic $a$). The Boolean ultrapower approach involves one arbitrary choice (the ultrafilter generating $^\star V[a]$). The Boolean-valued model or sheaf-theoretic approach involves no arbitrary choices, but only generates a model in a generalized sense. The corresponding methods of constructing differentially closed fields---$\Kc^M_a$, $\Kden/\mfr$, and $(\Kden,\nu)$---respectively, involve the same arbitrary choices and produce the same kind of model.}
  \label{fig:three-approaches}
\end{figure}

In the Boolean ultrapower approach to forcing, we build an elementary extension $^\star V$ of our starting universe of sets $V$ in such a way that it admits a forcing extension $^\star V[a]$ generated by a non-standard Cohen generic complex number $a$. $^\star V$ is a quotient of $V^{\mathrm{RO}(\Cb)}$ in the same way that $\Kden/\mfr$ is a quotient of $\Kden$ (quite literally, given \cref{prop:one-to-one-to-one}).

When I say `non-standard,' I mean it in the sense of non-standard analysis. $^\star V$ is the kind of object used in the \emph{non-standard superstructure} approach to non-standard analysis. Traces of this idea are visible in $\Kden/\mfr$, such as in the following proposition.

\begin{prop}\label[prop]{prop:subrings}
	For any ring $R \subseteq \Cb$, let $\underline{R}_{\dense}$ be the subring of $\Kden$ consisting of equivalence classes of locally constant functions taking values in $R$. For any such $R$, the ring $\underline{R}_{\dense}/(\mfr \cap \underline{R}_{\dense})$ has the same first-order theory as $R$.
\end{prop}

So, for example, in $\underline{\Rb}_{\dense}/(\mfr \cap \underline{\Rb}_{\dense})$, if $x\neq0$, then exactly one of $x$ and $-x$ has a square root and in $\underline{\Zb}_{\dense}/(\mfr \cap \underline{\Zb}_{\dense})$, if $x\neq 0$, then exactly one of $x$ and $-x$ is a sum of four squares. This is entirely unsurprising model-theoretically, given that this is a kind of generalized ultrapower construction, but it would be at least a little bit surprising if we were just considering the definition of $\Kden$ given in \cref{sec:main}.

So now we can finally summarize what $\Kden/\mfr$ is from the point of view of forcing. $\Kden/\mfr$ is the result of passing to a non-standard extension of the universe, $^\star V$, forcing to build a (non-standard) complex number $a$ Cohen generic over $^\star V$, and looking at the stalk $\Kc^{^\star V}_a$ of $^\star V$'s sheaf $\Kc^{^\star V}$ of (non-standard) meromorphic function at $a$. In this way, $\Kden/\mfr$ is the same kind of object as the stalk $\Kc^M_a$ from \cref{sec:re-reproving}.

\section{Generalizations and questions}
\label{sec:generalizations-and-questions}

One of the primary dividends of a simple semantic construction is that it is easy to generalize. What we have discussed in this paper is three perspectives of building semantics for the `(Cohen) generic first-order behavior' of a sheaf of differential rings, but we've focused on one particular sheaf of differential rings, $\Kc$, and there are many others. I'll refer to the theory produced by this common construction as the \emph{Cohen generic theory} of the sheaf in question. (As we'll see in a moment, this theory does not depend on the choice of Cohen generic point.)

For example, in \cite{Singer1978}, Singer studied the model companion, \CODF, of the theory of ordered differential fields. In Section~3 of \cite{Singer1978}, he showed that every finitely generated real\footnote{A field is \emph{(formally) real} if $-1$ is not a sum of squares.} differential field is isomorphic to a field of real-valued meromorphic functions on some open subset of $\Rb$. The proof is very similar to that of Seidenberg's embedding theorem, but Singer also notes at the end of Section~3 that the precise analog of \cref{lem:Seidenberg-lemma} fails, so we're left with the question of what precisely happens if we repeat the construction presented in this paper with the sheaf $\Rc$ of real-valued meromorphic functions on $\Rb$.

  Let $\Rden$ be the ring of real-valued meromorphic functions on $\Rb$ with dense open domain. One can show using the methods of this paper that for any maximal ideal $\mfr$ of $\Rden$, $\Rden/\mfr$ is a differential field whose field reduct is real-closed (and so has a unique compatible ordering, which is also first-order definable).

\begin{quest}
  Is $\Rden/\mfr$ a model of \CODF? Is it $\mathrm{NIP}$?
\end{quest}

It's not prima facie clear that the first-order theory of $\Rden/\mfr$ doesn't depend on the choice of $\mfr$. One can actually show this using the machinery of forcing fairly directly (using the idea of \emph{homogeneity} of forcing posets), but this goes a little bit beyond the aspects of forcing we've exposited in this paper.

\begin{prop}\label[prop]{prop:Rden-theory}
  The theory of $\Rden/\mfr$ does not depend on the choice of maximal ideal $\mfr$.
\end{prop}
\begin{proof}
  In the same way as \cref{sec:Boolean-valued}, define a language $\Lc(\Rden)$ extending the language of differential rings and define an $\RO(\Rb)$-valued truth valuation $\nu$ on $\Rden$ as an $\Lc(\Rden)$-structure. Any translation $\sigma : \Rb \to \Rb$ has a natural action on $\Rden$ by pre-composition, which extends to actions on $\RO(\Rb)$ and on $\Lc(\Rden)$. These actions are compatible in the sense that $\nu(\sigma \cdot \varphi) = \sigma \cdot \nu(\varphi)$ for any $\Lc(\Rden)$-sentence $\varphi$. Sentences in the language of differential rings are invariant under this action, so for any such sentence $\varphi$, we must have that $\nu(\varphi)$ is invariant under translations. The only elements of $\RO(\Rb)$ that are invariant under translations are $\varnothing$ and $\Rb$. Therefore the reduct of the theory $\{\varphi \in \Lc(\Rden) : \nu(\varphi) = \Rb\}$ to the language of differential rings is complete. Every structure of the form $\Rden/\mfr$ must be a model of this theory, so they are all elementarily equivalent.
\end{proof}

Interesting things happen semantically with the other constructions we've discussed here too. Given a countable transitive model $M$, let $\Rc^M$ be the subsheaf of $\Rc$ consisting of functions coded in $M$ (in the same manner as \cref{defn:M-meromorphic-sheaf}). For any real number $a$ Cohen generic over $M$, we know that the stalk $\Rc^M_a$ must be a model of the theory in \cref{prop:Rden-theory}, but we also know that it can't be a saturated model (unlike $\Kc^M_a$). This is because the field of constants in $\Rc^M_a$ is canonically a subfield of $\Rb$ and so in particular is Archimedean (in the ordered field sense). $\Rden/\mfr$ might be saturated by virtue of the fact that it is built using an ultrapower-like construction, but it's not clear whether this is true.

We can play a similar game with the sheaf of meromorphic functions on $\Cb^m$. The generic stalk of this is naturally a field with $m$ commuting derivations. There is a model companion of this theory, described by McGrail as $m$-\DCF{} in \cite{McGrail2000}. It seems likely that this builds a model of that model companion, but I'm fearful enough of functions of several complex variables to be wary of conjecturing this in print.

\begin{quest}
 Is $m$-\DCF{} the Cohen generic theory of the sheaf of meromorphic functions on $\Cb^m$?
\end{quest}

We can also consider other similar constructions, varying the sheaf or the field. For instance, we could think about the sheaf of smooth functions on $\Rb$, or meromorphic functions on the $p$-adics or on $\Fb_p((t))$ or on their algebraic closures.

\begin{quest}
  Is the Cohen generic theory of the sheaf of smooth functions on $\Rb$ the same as the Cohen generic theory of the sheaf of meromorphic functions on $\Rb$?
\end{quest}

\begin{quest}
  What can be said about the Cohen generic theories of stalks of sheaves of meromorphic functions on $\Qb_p$? On $\Cb_p$?
\end{quest}

In the case of $\Qb_p$ in particular, we know by the same general reasoning we used in \cref{sec:forcing} that whatever the resulting differential field is, its field of constants will be an elementary extension of $\Qb_p$. Like with $\Rb$, the topology on $\Qb_p$ is first-order definable, so we know that the resulting theory will still be strongly connected to the theory of $\Qb_p$.

\begin{quest}
  What is the relationship between the Cohen generic theory of the sheaf of meromorphic functions on $\overline{\Fb_p((t))}$ and $\mathsf{DCF}_p$, the theory of differentially closed fields of characteristic $p$?
\end{quest}

We can also ask what happens with notions of forcing that build reals other than Cohen forcing. For instance, random real forcing concerns itself with full measure sets rather than comeager sets. It seems less likely that this would be interesting in the context of meromorphic functions, but it might be interesting for smooth functions. In fact, the natural object in this context would arguably be the following:

\begin{defn}
 Fix a filter $F$ of subsets of $\Rb$ (or some other field on which derivatives make sense). A partial function $f : {\subseteq}\Rb \to \Rb$ is \emph{$F$-smooth} if for each $n$, the domain of the derivative $f^{(n)}$ is an element of $F$ (possibly with $\dom(f^{(n+1)})$ a strict subset of $\dom(f^{(n)})$).
\end{defn}

We can now repeat the entire story considering the ring of $F$-smooth functions modulo agreement on elements of $F$. Something along these lines (with $F$ the filter of full measure subsets of $\Rb$, for instance) might allow one to apply ideas from the model theory of differential fields to weak solutions of partial differential equations.

We also get additional structure on all of these objects, beyond just the structure of a ring with one or more derivations. Every entire holomorphic function $f : \Cb^n \to \Cb$ has a natural action on $\Kden/\mfr$. These can be interpreted pointwise on $\Kden$ and the quotient $\Kden/\mfr$ respects these functions. Moreover, the derivative interacts correctly with all of these, obeying the chain rule. For example, it makes sense to interpret the function $z \mapsto e^z$ on $\Kden/\mfr$ and it will satisfy $De^f = f'e^f$, as expected. (These are all facts which are not too hard to verify by hand but which are completely trivial thinking of $\Kden/\mfr$ in terms of forcing; in other words, it's obvious that we can compose elements of $\Kden/\mfr$ with holomorphic functions in a coherent way because elements of $\Kden/\mfr$ are morally just stalks of meromorphic functions at a point.)

Algebraically, this makes $\Kden/\mfr$ into something like an analytic version of a $C^\infty$-ring \cite{Moerdijk1991}. Model-theoretically, this ends up interpreting corresponding expansion of (an elementary extension of) $\Cb$, and these tend to not be terribly tame by themselves,\footnote{$(\Cb,+,\cdot,e^z)$ defines $\Zb \subseteq \Cb$, for instance. And similarly, $\Kden/\mfr$ expanded by the natural interpretation of $e^z$ defines the non-standard integers $^\star \Zb \subseteq {{}^\star \Cb}$.} but we can apply similar ideas to the other objects we've been discussing here.

It's known that the structure $\Rb_{\mathrm{an}}$ (i.e., $\Rb$ expanded with all real-analytic functions on unit cubes) is $o$-minimal. Given the germy nature of the construction we're doing, it makes sense to talk about real-analytic functions restricted to cubes on $\Rden/\mfr$. Given any $[g_1],\dots,[g_n] \in \Rden$ and any function $f : \Rb^n \to \Rb$ definable in $\Rb_{\mathrm{an}}$, the function $h(z) = f(g_1(z),\dots,g_n(z))$, while not meromorphic on all of its domain, is still meromorphic on a dense open set, and so can be interpreted as an element of $\Rden$. Again, the equivalence relation of agreeing on a dense open set respects this operation, as do all maximal ideals $\mfr$, so we can consider $\Rden/\mfr$ as a structure in the union of the language of differential rings and the language of $\Rb_{\mathrm{an}}$.

\begin{quest}
  Is the theory of $\Rden/\mfr$ expanded by the language of $\Rb_{\mathrm{an}}$ model-theoretically tame? Is it $\mathrm{NIP}$?
\end{quest}

\appendix

\section{$\Kden/\mfr$ is saturated}
\label{ap:saturated}

First, an easy observation.

\begin{prop}
  $|\Kden/\mfr| = 2^{\aleph_0}$.
\end{prop}
\begin{proof}
  The map taking a complex number $a$ to the constant function with value $a$ gives an injective ring homomorphism of $\Cb$ into $\Kden/\mfr$, so $|\Kden/\mfr| \geq 2^{\aleph_0}$. But $|\Kden| = 2^{\aleph_0}$, since the set of continuous functions on dense open subsets of $\Cb$ also has the cardinality of the continuum, implying that $|\Kden/\mfr| \leq 2^{\aleph_0}$ and so $|\Kden/\mfr| = 2^{\aleph_0}$.
\end{proof}

One thing to note is that while clearly the canonical embedding of $\Cb$ into $\Kden/\mfr$ is an embedding into the constants of $\Kden/\mfr$, not every constant is in the range of this embedding. (In the notation of \cref{prop:subrings}, the constants in $\Kden/\mfr$ correspond to the image of $\underline{\Cb}_{\dense} \subseteq \Kden$ under the quotient map.)

The argument to establish saturation is very similar in general structure to the argument that $\Kden/\mfr$ is a model of \DCF.\footnote{This is not really surprising, given that existential closure is itself a weak finitary kind of saturation.} First, we need an improvement of \cref{fact:indep-solution} part (\ref{solution}) that provides a perfect set of mutually generic solutions to the given differential polynomial.

\begin{fact}\label[fact]{fact:perf-soln}
  Let $U$, $p(x)$, $K \subseteq \Kc(U)$, $V \subseteq U$, $W \subseteq \Cb^n$, and the assignment of meromorphic functions $\vec{a} \mapsto f_{\vec{a}} : W \to \Kc(V)$ be as in \cref{fact:indep-solution} part (\ref{solution}). There is a non-empty perfect set $D \subseteq W$ such that for any distinct $\vec{a},\vec{b}_1,\dots,\vec{b}_m \in D$, $f_{\vec{a}}$ satisfies no non-zero differential polynomial of order less than $n$ with coefficients algebraic over the differential field generated by $\{h|_V : h \in K\} \cup \{f_{\vec{b}_1},\dots,f_{\vec{b}_m}\}$.
\end{fact}

Similarly, and more easily, we have the following.

\begin{fact}\label[fact]{fact:perf-soln-trans}
  Let $U$, $p(x)$, $K\subseteq \Kc(U)$, and $z_0$ be as in \cref{fact:indep-solution} part (\ref{generic-point}). For any $\vec{a} \in [0,1]^\omega$, let $g_{\vec{a}}$ be the unique holomorphic function on $U$ with Taylor series
  \(
  \sum_{n = 0}^\infty \frac{a_n}{n!}(z-z_0)^n.
  \)
  There is a non-empty perfect set $D \subseteq [0,1]^\omega$ such that for any distinct $\vec{a},\vec{b}_1,\dots,\vec{b}_m \in D$, $g_{\vec{a}}$ satisfies no non-zero differential polynomial with coefficients algebraic over the differential field generated by $\{h|_U : h \in K\} \cup \{g_{\vec{b}_1},\dots,g_{\vec{b}_m}\}$.
\end{fact}

One way to prove Facts~\ref{fact:perf-soln} and \ref{fact:perf-soln-trans} is by building an appropriate countable transitive model of \ZFC{} and then using the standard fact that one can build a perfect set of mutually Cohen generic reals over any countable transitive model of \ZFC{}.

Now we analogously upgrade \cref{lem:Seidenberg-dense}.

\begin{lem}\label[lem]{lem:perf-soln-sat}
  Let $U$ be a non-empty open subset of $\Cb$ and let $H$ be a countable set of meromorphic functions on $U$.
  \begin{enumerate}
    \begin{sloppypar}
  \item\label{soln-non-trans} Let $h_1,\dots,h_k \in H$ be nowhere identically vanishing meromorphic functions on $U$, and let $p(x;y_1,\dots,y_k)$ be a non-zero generic differential polynomial with order $n \geq 1$. Then there exists an open $W \subseteq U$ dense in $U$ and a family $(f_\alpha)_{\alpha \in 2^\omega}$ of meromorphic functions on $W$ such that
  \end{sloppypar}
    \begin{itemize}
    \item $p(f_\alpha;h_1|_W,\dots,h_k|_W)=0$ for each $\alpha \in 2^\omega$ and
    \item for any connected component $V$ of $W$, and any distinct $\alpha,\beta_1,\dots,\beta_m \in 2^\omega$, $f_\alpha|_V$ does not satisfy any non-zero differential polynomials of order less than $n$ with coefficients algebraic over the differential field generated by $\{h|_V : h \in H\} \cup \{f_{\beta_1}|_V,\dots,f_{\beta_m}|_V\}$.
    \end{itemize}
  \item\label{soln-trans} There exists a family $(g_\alpha)_{\alpha \in 2^\omega}$ of meromorphic functions on $U$ such that for any connected component $V$ of $U$, and any distinct $\alpha,\beta_1,\dots,\beta_m \in 2^\omega$, $g_\alpha|_V$ does not satisfy any non-zero differential polynomials with coefficients algebraic over the differential field generated by $\{h|_V : h \in H\} \cup \{g_{\beta_1}|_V,\dots,g_{\beta_m}|_V\}$.
  \end{enumerate}
\end{lem}
\begin{proof}
  For (\ref{soln-non-trans}), we largely repeat the proof of \cref{lem:Seidenberg-dense} in the positive order case: Let $(V_i)_{i < \omega}$ be an enumeration of a basis of open subsets of $U$ with the property that for every $n$, $\bigcup_{i > n} V_i$ is dense in $U$.
  Define the following data by induction: a sequence $(W_n)_{n < \omega}$ of pairwise disjoint connected non-empty open subsets of $U$ and  $(f_{\alpha,n})_{\alpha\in 2^\omega, n < \omega}$ such that $f_{\alpha,n}$ is meromorphic on $W_n$.

  At stage $n$, given $W_j$ for $j < n$, find the smallest $i < \omega$ such that $\bigcup_{j < n} W_j$ is not dense in $V_i$. Let $O$ be a connected component of $V_i \setminus \cl \bigcup_{j < n} W_j$. By \cref{fact:perf-soln}, we can find a connected open subset $W_n \subseteq O$ and a family of meromorphic function $(f_{\alpha,n})_{\alpha \in 2^\omega}$ on $W_n$ satisfying the two bullet points in part (\ref{soln-non-trans}) of the statement of the lemma.

  Let $W = \bigcup_{n < \omega}W_n$. For each $\alpha \in 2^\omega$, let $f_\alpha = \bigcup_{n < \omega}f_{\alpha,n}$ be the unique meromorphic function on $W$ satisfying that $f_\alpha|_{W_n} = f_{\alpha,n}$ for each $n < \omega$. $(f_\alpha)_{\alpha \in 2^\omega}$ now satisfies the required conditions by construction.

  The proof of (\ref{soln-trans}) is similar, using \cref{fact:perf-soln-trans} in place of \cref{fact:perf-soln}.
\end{proof}

Finally, we can prove \cref{thm:main}.

\begin{proof}[Proof of {\cref{thm:main}}]
  Fix a finitely generated field $K \subseteq \Kden/\mfr$. We will argue that for any non-algebraic type $t(x) \in S_1(\acl(K))$, $\Kden/\mfr$ realizes a continuum-sized Morley sequence in $t$. Since $\DCF$ is $\aleph_0$-stable, this is enough to establish that $\Kden/\mfr$ is $2^{\aleph_0}$-saturated \cite[Prop.~7.21]{Pillay1983}.

  If realizations of $t(x)$ satisfy any non-zero differential polynomial with coefficients algebraic over $K$, then $t(x)$ is the unique generic type satisfying some differential polynomial $p(x)$ with coefficients in $K$. Let $p(x) = p_0(x;[[h_0]],\dots,[[h_k]])$ for some generic differential polynomial $p_0$ with each $[[h_i]]$ non-zero. Let $H$ be a finite set of meromorphic functions containing $\{h_0,\dots,h_k\}$ and such that $\{[[h]] : h \in H\}$ generates $K$. By the same argument as in the proof of \cref{thm:models-DCF}, we can find an open set $U$ with $1-[\onebf_U] \in \mfr$ such that each $h_i$ is nowhere identically vanishing on $U$. By shrinking $U$ (and replacing elements of $H$ with their restrictions to $U$) if necessary, we may assume without loss of generality that the domain of every element of $H$ is $U$. Applying \cref{lem:perf-soln-sat} part (\ref{soln-non-trans}) (with this choice of $H$), we now get an open $W \subseteq U$ dense in $U$ and a family $(f_\alpha)_{\alpha \in 2^\omega}$ of meromorphic functions satisfying the conditions in \cref{lem:perf-soln-sat}. \cref{lem:basic-stuff} now implies that for any distinct $\alpha,\beta_0,\dots,\beta_m \in 2^\omega$, $[[f_\alpha]]$ realizes the unique non-forking extension of $t(x)$ over $\acl(K\cup\{[[f_{\beta_0}]],\dots,[[f_{\beta_m}]]\})$ (since it is a solution of $p(x)$ that satisfies no differential polynomial of lower order with coefficients from that set). This implies that $([[f_\alpha]])_{\alpha \in 2^\omega}$ is a Morley sequence in $t$ over $K$.

  If realizations of $t(x)$ do not satisfy any non-zero differential polynomials with coefficients algebraic over $K$, we can similarly use \cref{lem:perf-soln-sat} part (\ref{soln-trans}) to get a Morley sequence $([[g_\alpha]])_{\alpha \in 2^\omega}$ in $t$ over $K$.

  Finally, it is a basic result of model theory that any two elementarily equivalent saturated structures of the same cardinality are isomorphic, implying that $\Kden/\mfr \cong \Kden/\mfr'$ for any maximal ideal $\mfr'$.
\end{proof}

\end{document}